\documentclass[a4paper,11pt, reqno]{amsart}
\usepackage{tikz}
\usepackage{graphicx}
\usepackage{array}
\usepackage{geometry}
\usetikzlibrary{cd}
\usepackage{}
\usepackage{amssymb, amsmath, amsfonts, amscd}
\usepackage{mathrsfs, latexsym, array, longtable}
\usepackage{xcolor, bm, enumitem}
\usepackage{hyperref}
\usepackage{booktabs}

\hypersetup{colorlinks=true,citecolor=red,linkcolor=blue,urlcolor=black,}

\theoremstyle{definition}
\newtheorem{theorem}{Theorem}[section]
\newtheorem{proposition}[theorem]{Proposition}
\newtheorem{lemma}[theorem]{Lemma}

\newtheorem{conjecture}[theorem]{Conjecture}
\newtheorem{remark}[theorem]{Remark}
\newtheorem{definition}[theorem]{Definition}
\newtheorem{example}[theorem]{Example}
\newtheorem{corollary}[theorem]{Corollary}

\DeclareMathOperator{\Aut}{Aut}
\DeclareMathOperator{\Alb}{Alb}
\DeclareMathOperator{\vol}{vol}

\DeclareMathOperator{\Pic}{Pic}

\DeclareMathOperator{\im}{im}

\DeclareMathOperator{\GL}{GL}
\newcommand{\C}{\mathbb C}
\newcommand{\Z}{\mathbb Z}
\newcommand{\N}{\mathbb N}
\newcommand{\Q}{\mathbb Q}
\newcommand{\Pp}{\mathbb P}
\newcommand{\OO}{\mathcal O}

\newcommand{\Iv}{\operatorname{Ivol}}

\title{\textbf{On irregular threefolds of Kodaira dimension two}}
\author{Ziqi Liu \and Pengjin Wang}
\date{\today}
\address{\rm Ziqi Liu, Shanghai Center for Mathematical Sciences, Fudan University, Jiangwan Campus, Shanghai, 200438, China}
\email{zqliu22@m.fudan.edu.cn}
\address{\rm Pengjin Wang, Shanghai Center for Mathematical Sciences, Fudan University, Jiangwan Campus, Shanghai, 200438, China}
\email{pjwang25@m.fudan.edu.cn}
\begin{document}
\pagestyle{myheadings}
\begin{abstract}
    Let $X$ be an irregular threefold of Kodaira dimension two. We show that the Iitaka volume $\Iv(X)$ of $X$ is at least $\frac{1}{100}$ and $\lvert mK_X\rvert$ defines the Iitaka fibration for $m\geq 32$. Both $\frac{1}{100}$ and $32$ are sharp.
\end{abstract}
\markboth{\hfill Ziqi Liu and Pengjin Wang\hfill}{\hfill Irregular threefolds of Kodaira dimension two\hfill}

\maketitle
\tableofcontents
\addtocontents{toc}{\protect\setcounter{tocdepth}{1}}
\section{Introduction}
Classifying algebraic varieties is one of the most important goals in birational geometry. Curves are classified by their genus into rational curves, elliptic curves and curves of general type. In dimension two, the Enriques-Kodaira classification describes surfaces up to birational equivalence according to their Kodaira dimension. In higher-dimensional geometry, Kodaira dimension continues to play a central role in the birational classification of algebraic varieties, although a complete classification remains far out of reach. 

\subsection{Effective Iitaka fibration}
Let $X$ be a smooth projective variety of Kodaira dimension $\kappa(X)\geq 0$. By Iitaka's construction, there exists a smooth projective birational model $X^\prime $ of $X$ admitting a contraction $f: X^\prime\to Z$, where $Z$ is a normal projective variety of dimension $\kappa(X)$ and the general fiber of $f$ has Kodaira dimension zero. The induced map $X \dashrightarrow Z$ is called an Iitaka fibration of $X$ and is uniquely determined up to birational equivalence. For any sufficiently divisible integer $m\geq 1$, the pluricanonical system $\lvert mK_X\rvert$ defines the Iitaka fibration. The effective Iitaka fibration conjecture predicts the existence of a uniform integer depending only on the dimension and Kodaira dimension:

\begin{conjecture}[{\cite[Conjecture 1.7]{HaconMcKernan2006}}]\label{conj: eff-iitaka}
    Let $X$ be a smooth projective variety of dimension $d$ and Kodaira dimension $\kappa(X)\geq 0$. Then there is a positive integer $m(d, \kappa)$, depending only on $d$ and $\kappa$, such that the pluricanonical system $\lvert mK_X \rvert $ defines the Iitaka fibration for all positive integers $m$ divisible by $m(d,\kappa)$.

Moreover, when $\kappa=d$ or $d-1$, one expects $\lvert mK_X \rvert $ to define the Iitaka fibration for all sufficiently large $m$, without any divisibility assumption.
\end{conjecture}

When $\kappa=d$, the conjecture was proved by H. Tsuji, S. Takayama, C. D. Hacon and J. M\textsuperscript{c}Kernan in \cite{Tsuji06,Takayama06,HaconMcKernan2006}. When $\kappa=d-1$, the conjecture was proved in \cite{Jiang2013}. Note that the divisibility condition in \cite{Jiang2013} can be removed by applying \cite[Lemma 2.3.4]{HaconMcKernanXu2013} on the Iitaka base. For $0\leq \kappa\leq d-2$, the divisibility condition is still necessary. Further progress towards the conjecture was made in \cite{ViehwegZhang09,BirkarZhang16}. However, the bounds obtained there also depend on topological invariants of the general fiber. Recently, OpenAI \cite{OpenAI2026Iitaka} claimed a proof of the conjecture in full generality. Our arguments do not rely on this work.

The remaining problem is to determine explicit uniform bounds. For a $d$-dimensional smooth projective variety $X$ of Kodaira dimension $\geq d-1$, define
\[r_s(X):=\min\{l\in \mathbb{Z}\vert \lvert mK_X\rvert\, \text{defines the Iitaka fibration for all}\, m\geq l\}.\]
Extending the notion of the canonical stability index in \cite{ChenLiu2024}, we define 
\[r_d:=\sup\{r_s(X)\vert X \,\text{is a smooth projective}\,d\text{-fold} \,\text{of general type} \},\]
\[\hat{r}_{d}:=\sup\{r_s(X)\vert X \,\text{is a smooth projective}\,d\text{-fold} \,\text{with}\,\kappa(X)=d-1\}.\]
When $d=2$, we have $r_2=5$ and $\hat{r}_{2}=14$ by \cite{Bombieri1973,KatsuraUeno1985}. The optimal value of $m(2,0)$ is $12$ by \cite{Enriques1914I}. When $d=3$, we have that $r_3\leq 57$ by \cite{ChenChen2015, Chen2018}. Also, the optimal value of $m(3,0)$ is determined by \cite[Corollary 5]{MachidaOguiso98}. For threefolds with $\kappa=1$ and $2$, partial results were obtained by H. Chen \cite{Chen2021} and A. T. Ringler \cite{Ringler2007}.

In this article, we study smooth irregular threefolds of Kodaira dimension two, with particular emphasis on the relationship between the Albanese morphism and the Iitaka fibration. For an irregular surface (resp., threefold) $X$ of general type, $\lvert mK_X \rvert$ defines a birational map for $m\geq 3$ (resp., $m\geq 5$) by \cite{Bombieri1973, CCCJ21}. For an irregular surface of Kodaira dimension one, $\lvert mK_X \rvert$ defines the Iitaka fibration for $m\geq 14$ by \cite{KatsuraUeno1985}, which coincides with $\hat{r}_2$. The following theorem generalizes the result to dimension $3$ and establishes a sharp effective bound in this setting:

\begin{theorem}\label{mainthm: eff-iitaka}
    Let $X$ be a smooth projective irregular threefold with $\kappa(X)=2$. Then
\begin{enumerate}[label=(\arabic*), ref=\arabic*]
    \item\label{item: eff-iitaka-1} $|mK_X|$ defines the Iitaka fibration for any $m\geq 32$;  
    \item\label{item: eff-iitaka-2} $|mK_X|$ defines the Iitaka fibration for any $m\geq 14$, if the Albanese map factors through the Iitaka fibration.
\end{enumerate}
Both bounds $14$ and $32$ are optimal.
\end{theorem}

For an irregular smooth projective variety $X$ of general type and dimension $d$, M. Chen and Z. Jiang conjectured in \cite[Conjecture 4.8]{ChenJiang2022} that $\lvert mK_X \rvert$ induces a birational map for all $m\geq r_{d-1}$. This conjecture has been proven in dimension $\leq 3$, by \cite{Bombieri1973,CCCJ21}.

Theorem \ref{mainthm: eff-iitaka}(\ref{item: eff-iitaka-2}) motivates  the following analogous lifting principle:

\begin{conjecture}\label{conj:lifting-principle for kod n-1}
    Let $X$ be a smooth projective variety of dimension $d\geq 3$ and Kodaira dimension $d-1$. If $q(X)>0$ and the Albanese map factors through the Iitaka fibration, then $\lvert mK_X \rvert$ defines the Iitaka fibration for all $m\geq \hat{r}_{d-1}$.
\end{conjecture}

\subsection{The minimal Iitaka volume}
Let $X$ be a smooth projective variety of dimension $d$ and Kodaira dimension $0<\kappa\leq d$. The Iitaka volume of $X$ is defined as 

\begin{equation*}
  \Iv(X)=\limsup_{m\to +\infty} \frac{h^0(X,mK_X)}{m^{\kappa}/\kappa!}.
\end{equation*}

When $\kappa=d$, the Iitaka volume coincides with the usual canonical volume of $X$. C. D. Hacon and J. M\textsuperscript{c}Kernan \cite{HaconMcKernan2006} proved that canonical volumes of smooth projective varieties of general type with fixed dimension admit a uniform positive lower bound. Canonical volume has since played a central role in boundedness and moduli theory of varieties of general type \cite{HaconMcKernanXu2013,HaconMcKernanXu2014,HaconMcKernanXu2018,Kollar2023}. 

When $\kappa<d$, it is natural to seek analogous boundedness results and develop a corresponding moduli theory. The following conjecture is an analogue of \cite[Corollary 1.3]{HaconMcKernan2006}:

\begin{conjecture}\label{conj: min-iitaka-vol}
    Let $X$ be a smooth projective variety of dimension $d$ and Kodaira dimension $\kappa(X)> 0$. Then there is a positive real number $\epsilon(d, \kappa)$, depending only on $d$ and $\kappa$, such that $\Iv(X)\geq \epsilon(d, \kappa)$.
\end{conjecture}

Note that Conjecture \ref{conj: min-iitaka-vol} follows from Conjecture \ref{conj: eff-iitaka}, which is well-known to experts. Z. Li formulated a stronger conjecture \cite[Conjecture 1.6]{Li2024}: for a fixed dimension $d$ and a fixed DCC set $I\subset [0,1] \cap \Q$, the Iitaka volumes of projective klt pairs $(X,B)$ with $\operatorname{dim}X=d$ and coefficients of $B$ in $\mathcal{I}$ satisfy DCC. This more general conjecture holds in dimension $\leq 3$ by \cite[Theorem 8.1]{ProkhorovShokurov2009} and \cite[Corollary 4.4]{Li2024}. G. Chen, J. Han and W. Liu generalized the conjecture from klt pairs with $\Q$-coefficients to lc pairs with $\mathbb{R}$-coefficients in \cite[Conjecture 1.1]{ChenHanLiu2025} and proved it in dimension at most three in \cite[Theorem 1.2]{ChenHanLiu2025}. C. Birkar proved that the set of Iitaka volumes of lc-trivial fibrations with an ample $\Z$-divisor of fixed volume on the general fibers satisfies DCC in \cite[Theorem 1.6]{Birkar2021}. J. Jiao also proved the DCC property under additional assumptions in \cite{Jiao2025}.

In particular, Conjecture \ref{conj: min-iitaka-vol} holds in dimension $\leq 3$. It remains to determine the optimal lower bounds. For this purpose, define \[v(d,\kappa)=\inf\{\Iv(X)\,\vert\, X\,\text{ is a smooth projective}\,d\,\text{-fold with}\, \kappa(X)=\kappa  \}.\] When $d=2$, the canonical bundle formula gives $v(2,1)=\frac{1}{6}$; see \cite[Theorem 1.7(0)]{ChenHanLiu2025}. It is also well-known that $v(2,2)=1$. When $d=3$, by \cite{ChenChen2015} we have $\frac{1}{420}\geq v(3,3)\geq \frac{1}{1680}$ while there is still a lack of explicit values for $v(3,1)$ and $v(3,2)$. As a step towards determining $v(3,2)$, we establish the following sharp lower bound for irregular threefolds of Kodaira dimension two:

\begin{theorem}\label{mainthm: iitaka-vol}
    Let $X$ be a smooth projective irregular threefold with $\kappa(X)=2$. Then
    \begin{enumerate}

        \item\label{item2 for mainthm2} $\Iv(X)\geq \frac{1}{100}$;
        \iffalse
        if the Albanese map does not factor through the Iitaka fibration.
        \fi
        \item $\Iv(X)\geq \frac{1}{24}$, if the Albanese map factors through the Iitaka fibration. 
 \end{enumerate}
 Moreover, the bound $\frac{1}{100}$ is optimal.
\end{theorem}
The minimal Iitaka volume for irregular surfaces of Kodaira dimension one is equal to $v(2,1)=\frac{1}{6}$ by Example \ref{eg-1-6}. Theorem \ref{mainthm: iitaka-vol} is an analogue of \cite[Theorem 1.5]{Jiang2021}, which gives a lower bound $\frac{3}{8}$ for the canonical volume of irregular threefolds of general type. 

We also characterize the threefolds attaining the extremal bounds in Theorem \ref{mainthm: eff-iitaka} and Theorem \ref{mainthm: iitaka-vol}.
\begin{theorem}\label{mainthm: eqcaseiveff}
Let $X$ be a smooth projective irregular  threefold with $\kappa(X)=2$. 
The following conditions are equivalent.
\begin{itemize}
    \item[(1)] $\Iv(X)=\frac{1}{100}$;
    \item[(2)] $|31K_X|$ does not define the Iitaka fibration;
    \item[(3)] $h^0(X,31K_X)=1$;
    \item[(4)] $X\sim_{\text{bir}} (T\times E)/G$ is generically isotrivial (see subsection \ref{subsec: geneiso}), where  
    \[
 T\cong \{w^2=x^{10}+y^{10}+z^{5}\}\subset\Pp(1,1,2,5),\qquad G\cong \Z_{10} \times \Z_{10},
\]
and for any $g=(s,t)\in \Z_{10} \times \Z_{10},[x:y:z:w]\in T, e \in E$, one has
\[
 g.([x:y:z:w],e)=
([\zeta_{10}^sx:\zeta_{10}^ty:z:w],e+t_g),
\]
where $G$ acts on $E$ by translations $g: e \mapsto e+t_g$.
\end{itemize}
\end{theorem}

Indeed, Theorem \ref{mainthm: iitaka-vol} in the non-Albanese-factoring case follows from the following stronger result, which does not require $X$ to be irregular:

\begin{theorem}\label{mainthm:translations}
    Let $X$ be a smooth projective threefold of Kodaira dimension two. Assume that $X$ is generically isotrivial in the sense of Section \ref{subsec: geneiso} with $G$ abelian.
    
    Then the Iitaka volume $\Iv(X)\geq \frac{1}{100}$, where the bound  $\frac{1}{100}$ is optimal.
\end{theorem}
Combining Theorem \ref{mainthm:translations} with Hao-Schreieder's theorem (Theorem \ref{thm:hao-schreieder}), we obtain the lower bound $\frac{1}{600}$ for the Iitaka volume of every generically isotrivial threefold of Kodaira dimension two, since $G$ has a translation subgroup of index at most $6$. However, it seems far from being optimal. Theorem \ref{mainthm:translations} is closely related to bounds on the orders of finite abelian automorphism groups of surfaces of general type. Consider the following example:

\begin{example}\label{eg-1-6}
Let $D$ be the smooth projective model of $y^2=x^6-1$, which has genus two. Let \[G:=<\sigma, \tau>\cong \mathbb{Z}/6\mathbb{Z}\oplus \mathbb{Z}/2\mathbb{Z},\; \sigma(x,y)=(\zeta_6 x, y), \; \tau(x,y)=(x,-y),\] where $\zeta_6$ is a primitive sixth root of unity. Let $E$ be an elliptic curve on which $G$ acts faithfully by translations, and set $S:=(D\times E)/G$, where the action on the product is diagonal. One can easily check that \[\Iv(S)=\frac{2g(D)-2}{\lvert G\rvert}=\frac{1}{6},\] which is the minimal Iitaka volume of all surfaces of Kodaira dimension one.

More generally, \cite[Theorem 1]{nakajima1987abelian} gives $\lvert G\rvert\leq 4g+4$ for every finite abelian subgroup $G\subset \operatorname{Aut}(C)$, where $C$ is a smooth projective curve of genus $g\geq 2$. Thus, for any surface $S^\prime$ birational to $(C\times E)/G$, where $C$ is a curve of genus $g\geq 2$, $E$ is an elliptic curve and $G$ acts as translations on $E$ and diagonally on $C\times E$, we have \[\operatorname{Ivol}(S^\prime)=\frac{2g-2}{\lvert G \rvert}\geq \frac{g-1}{2(g+1)}\geq \frac{1}{6}.\]
\end{example}

\subsection{Sketch of the proof} 
By \cite[Theorem 4.4]{ChenChurchHao2025}, smooth irregular threefolds $X$ of Kodaira dimension two fall into two classes. In the first, the Albanese morphism factors through the Iitaka fibration. In the second, $X$ is birational to $(S\times E)/G$, where $S$ is a surface of general type, $E$ is an elliptic curve, and $G$ acts as translations on $E$ and diagonally on $S\times E$.

For the first class, we establish a point-separation theorem for varieties of intermediate Kodaira dimension (Theorem \ref{thm:sep-inter-kodaira}), extending \cite[Theorem 2.8]{ChenHacon2007}. Combining this with an explicit analysis of pluricanonical systems on surfaces of Kodaira dimension one (Lemma \ref{lem: surface-4-6-10}), we obtain the effective bound in Theorem \ref{thm: 14K}. To bound the Iitaka volume, we distinguish two cases according to the dimension of a general Albanese fiber. When this dimension is two, we follow the strategy of \cite{ChenChen2008}, and seek a constant $c>0$ such that $K_X-cF$ is $\Q$-effective, where $F$ is the general fiber. In Theorem \ref{thm: non-iso-alb-fiber-two}, we obtain the required nonvanishing by studying the slopes of the pluricanonical pushforwards $a_\ast \omega_X^{\otimes m}$. When the general Albanese fiber has dimension one, the effective Iitaka bound gives $P_6\geq 4$, and the geometry of the $6$-canonical map yields $\Iv(X)\geq \frac{1}{12}$. To obtain a stronger bound, we use the method of \cite{Jiang2021}. In Theorem \ref{thm: non-iso-alb-fiber-one}, we perform finite \'{e}tale base change induced by isogenies of the Albanese variety. The Iitaka volume scales by the degree of the cover, and the IT$_0$ property ensures that $P_2$ scales in the same way. This yields the stronger bound $\frac{1}{4}$.

For the second class, $X$ is generically isotrivial characterized in subsection \ref{subsec: geneiso}.
Under this assumption, we obtain Theorem \ref{thm: geneisotrivialvollowerbd} for minimal Iitaka volume $\frac{1}{100}$ as follows.
Corollary \ref{cor: chargeneisotrivial3fold} identifies the Iitaka volume of $X$ with $\frac{V}{|G|}$, where $V$ denotes the volume of a smooth minimal surface $S$ of general type and $G\leq \Aut(S)$ is abelian.
This can be bounded from below by applying methods in \cite{lu2025sharp} via considering $G$-invariant pluricanonical pencils for such $G$ faithfully acting on $S$.
If this pencil is rational, then we apply \cite[Theorem 3.1]{lu2025sharp} in Proposition \ref{prop: ratpencilcase}.
If this pencil is irrational, then we apply slope inequalities and analyze the group action on elliptic curves as in Proposition \ref{prop: irratpencilcase}.
These arguments yield the sharp lower bound of $\frac{V}{|G|}$ and the Iitaka volume of $X$.
Moreover, for the extremal case $\frac{V}{|G|}=\frac{1}{100}$, we can characterize canonical model of $S$ and $G$ by Lemma \ref{lem: whenT12G100}.

To establish the effective Iitaka bound in Theorem \ref{thm: geneisotrivialeffIitakafib},
we first prove Proposition \ref{prop: nVbdimply32K}, which bounds the degree of the restriction of natural quotient map $S\to S/G$ to curves passing through very general points. 
These estimates satisfy the positivity of intersection numbers listed in Lemma \ref{lem: criterionforsepa} and we can apply this lemma.
Combined with more efficient bounds (Proposition \ref{prop: ratpencilcase}, Lemma \ref{lem: pgq1bdG}) of $\frac{|G|}{V}$, this settles most cases.
It remains to consider surfaces
with $q(S)=0$ and $V=1$, together with
$(2,3)$-surfaces.

For $(1,1)$-surfaces,  the description of the canonical ring and a lattice argument in Proposition \ref{prop: pg1-volume1-abelian} give
$|G|\leq 36$, after which
Lemma \ref{lem: Godsurf} applies.

For $(1,0)$-surfaces, Proposition \ref{prop: ratpencilcase} implies  $|G|\leq 70$. 
An argument similar to that of Proposition \ref{prop: nVbdimply32K} formulated in Lemma \ref{lem: Godsurf} handles all but a few cases.
Then we exclude those cases by analyzing the genera of
curves passing through very general points
and the induced group actions on their
normalizations.

For $(1,2)$ and $(2,3)$-surfaces, we use the known
descriptions of their canonical rings  and construct $G$-invariant pluricanonical global sections whose ratios generate $\C(S)^G$.
Then we obtain the effectivity of the Iitaka fibration by Lemma \ref{lem: reducemKbirattoS/G} (2).

Finally, for \((1,2)\)-surfaces with $|G|=100$,
Lemma~\ref{lem: whenT12G100} determines the canonical model
and the group action.
Then we compute the
pluricanonical maps of $X$ explicitly in
subsection~7.1, completing the proof and
establishing the optimality of the bound.

\section*{Acknowledgments} 
The authors appreciate Prof. Meng Chen and Prof. Chen Jiang for their support and encouragement. They also thank Prof. Xin L\"u and Dr. Songsong Huang for helpful discussions. Part of this work was done when the authors attended the Summer School ``Algebraic Geometry in Positive Characteristic''
held at Xiamen
University in July 2026. The authors are grateful to the organizers for their  hospitality. This work was supported by National Key Research and Development Program of China \#2023YFA1010600. 
This work was supported by NSFC for Innovative Research Groups (No. 12121001).

The main framework of this article was developed by the authors. The authors use ChatGPT to undertake some technical computations related to surfaces, improve the bound, suggest examples and polish the English. 
All arguments and computations were verified by the authors, who take full responsibility for the content of this article.

\section{Preliminaries}
\subsection{Notation and Conventions}
A \emph{fibration} is a projective morphism $f:X\to Z$ between normal varieties such that $f_*\OO_X=\OO_Z$.
In particular, $f$ is surjective with connected fibers.

For a smooth projective variety $X$, we denote by $\Pic^0(X)$ the abelian variety parametrizing line bundles algebraically equivalent to zero, and by $a_X: X\to \Alb(X)$ its Albanese morphism. We identify $\Pic^0(\Alb(X))$ with $\Pic^0(X)$ via pullback by $a_X$.

For a $\Q$-divisor $D$ on a normal projective variety $X$, $H^0(X,D)$ is defined to be $H^0(X, \lfloor D \rfloor)$.

For a positive integer $i$,
denote by $\zeta$ (resp. $\zeta_i$) the root of unity (resp. the root of unity with order $i$).
\subsection{Iitaka volume and Iitaka fibration}
Let $X$ be a projective variety with at worst $\Q$-factorial terminal singularities.
\begin{itemize}
    \item[(1)] If $|mK_X|\neq \emptyset$ for some positive integer $m$,
    the \emph{Kodaira dimension} of $X$ is defined by 
\[
 \kappa(X):= 
 \max \{\dim\Phi_{|mK_X|}(X)\mid m\in\mathbb{Z}_{\geq1}\},
\]
where $\Phi_{|mK_X|}$ is the rational map defined by $|mK_X|$. Otherwise, $\kappa(X):= -\infty$; 

    \item[(2)] If $\kappa(X)\geq 0$, then the \emph{Iitaka volume} of $X$ is defined as 
    \begin{equation*}
  \Iv(X)=\limsup_{m\to +\infty} \frac{h^0(X,mK_X)}{m^{\kappa}/\kappa!};
\end{equation*}
    \item[(3)] If $\kappa(X)>0$ and $K_X$ is semi-ample, then
    the \emph{Iitaka fibration} of $X$ is defined as a fibration     
\[
f:X\to Z
\]
with 
an ample $\mathbb{Q}$-divisor $L$ on $Z$ such that $K_X\sim_{\mathbb{Q}}f^*L$.
In this case,
we have $\dim Z=\kappa(X)$ and
$\Iv(X)=L^{\dim Z}$;
\item[(4)] If $\kappa(X)> 0$,
then we say that the linear system $|mK_X|$ \emph{defines the Iitaka fibration}, if 
there is a birational map
\[
\varphi_m:\Phi_{|mK_X|}(X) \dashrightarrow Z, \qquad  \varphi_m \circ\Phi_{|mK_X|} =f
\]
where $Z$ is the base of the Iitaka fibration $f$;
\item[(5)] The \emph{irregularity} of $X$ is defined as 
    \[q(X):=h^1(X,\OO_X).\]
    If $q(X)>0$, then we say $X$ is \emph{irregular}.
\item[(6)] For $m\geq 1$, the \emph{$m$-th plurigenus} of $X$ is defined by $P_m(X):=h^0(X, mK_X)$.
\end{itemize}
\begin{lemma}\label{lem: lower bound for volume}
    Let $S$ be a normal projective surface and $D$ be a $\Q$-Cartier divisor on $S$. If the rational map $f$ defined by $D$ has two-dimensional image, then $\operatorname{vol}(D)\geq h^0(S, D)-2$.
\end{lemma}
\begin{proof}
    Choose a resolution $\pi: S^\prime\to S$ such that $\operatorname{Mov}\lvert \pi^\ast D\rvert $ is free. We have the following diagram:
    \begin{center}
        \begin{tikzcd}
            S \arrow["f"]{r} & \Pp^{h^0-1}\\
            S^\prime \arrow["\pi"]{u} \arrow["\phi"]{r} & S_0, \arrow[hook]{u}
        \end{tikzcd}
    \end{center}
    where $S_0$ is the image of $f\circ \pi$ in $\Pp^{h^0-1}$ and $\phi$ is generically finite. Let $M$ be a general member of $\operatorname{Mov}\lvert \pi^\ast D\rvert$. Then \[\operatorname{vol}(D)\geq (M^2)= \operatorname{deg} (\phi)\cdot \operatorname{deg}_{\Pp^{h^0-1}}S_0\geq h^0(S,D)-2,\]
    where the last inequality follows from \cite[Lemma 1.3]{Kobayashi1992}.
\end{proof}
\begin{lemma}\label{lem: debarre}
    Let $f: X_0\to S_0$ be a fibration between normal projective varieties and $\pi:X\to X_0$ be a birational morphism. Let $D$ be a $\Q$-Cartier nef divisor on $S_0$ and $E$ be an effective $\pi$-exceptional divisor on $X$. Then $h^0(X, \pi^\ast f^\ast D+E)=h^0(S_0, D)$ and the map induced by $\pi^\ast f^\ast D+E$ on $X$ descends to $S_0$.
\end{lemma}
\begin{proof}
    The proof of \cite[Lemma 7.1]{Debarre2001} applies equally to $\Q$-divisors.
\end{proof}

\subsection{Vector bundles on elliptic curves}\label{subsec: HNfiltration}
Let $E$ be an elliptic curve. For a vector bundle $\mathcal{E}$ on $E$, we may write \[ \mathcal{E}=\bigoplus_{i=1}^{n} \mathcal{E}_i,\] where each $\mathcal{E}_i$ is indecomposable, that is, it cannot be written as a direct sum of nonzero subbundles. By \cite[Page 3, Fact]{Tu1993}, every $\mathcal{E}_i$ is semistable. Thus, it is easy to obtain the Harder-Narasimhan filtration of $\mathcal{E}$. Indeed, we may further assume that \[\mu(\mathcal{E}_1)=\cdots=\mu(\mathcal{E}_{i_1})>\cdots >\mu(\mathcal{E}_{i_{k-1}+1})=\cdots =\mu(\mathcal{E}_{i_k}),\]
where $\mu(\mathcal{E}_i)=\frac{\operatorname{deg}\mathcal{E}_i}{\operatorname{rank}\mathcal{E}_i}$ for each $i$ and $i_k=n$. Then the Harder-Narasimhan filtration is just \[0=\mathcal{F}_0\subsetneq \mathcal{F}_1 \subsetneq \cdots \subsetneq \mathcal{F}_k=\mathcal{E}, \]
where $\mathcal{F}_l=\mathcal{E}_1\oplus \cdots \oplus \mathcal{E}_{i_l}$ for any $1\leq l\leq k$.
\begin{definition}
    For a vector bundle $\mathcal{E}$ on $E$ as above, the \emph{minimal} (resp., \emph{maximal}) \emph{slope} $\mu_{\min}(\mathcal{E})$ (resp., $\mu_{\max} (\mathcal{E})$) of $\mathcal{E}$ is defined to be $\min\{\mu(\mathcal{E}_i)\vert 1\leq i\leq n\}$ (resp., $\max\{\mu(\mathcal{E}_i)\vert 1\leq i\leq n\}$).
\end{definition}
\begin{lemma}[{\cite[Lemma 2.6, Lemma 2.8]{HuZhang2025}}]\label{lem:HuZhangSlope}
    Let $V_1, V_2$ be two vector bundles on an elliptic curve $E$. Then,
    \begin{enumerate}
        \item $\mu_{\min}(V_1\otimes V_2)=\mu_{\min}(V_1)+\mu_{\min} (V_2)$.
        \item If there is a generically surjective morphism $V_1 \to V_2$, then $\mu_{\min}(V_2)\geq \mu_{\min} (V_1)$. 
    \end{enumerate}
\end{lemma}
\begin{lemma}\label{lem:slope-nonvanishing}
    Let $E$ be an elliptic curve and $V$ be a vector bundle on $E$. If $\mu_{\max}(V)\geq 0$, then there exists a line bundle $D\in \Pic^0(E)$ such that $h^0(V\otimes D)\neq 0$.
\end{lemma}
\begin{proof}
    Let $V_0$ be an indecomposable direct summand with the maximal slope. If $\operatorname{deg}V_0>0$, it suffices to apply the Riemann-Roch formula. If $\operatorname{deg}V_0=0$, the conclusion follows from \cite[Theorem 5]{Ati57}.
\end{proof}

\subsection{Pushforwards under morphisms to abelian varieties} We recall some basic notions and results concerning pushforwards of pluricanonical bundles under morphisms to abelian varieties.

\begin{definition}
    Let $\mathcal{F}$ be a coherent sheaf on an abelian variety $A$.
    
    We say that $\mathcal{F}$ satisfies \emph{IT$_0$} if $H^i(A, \mathcal{F}\otimes P)=0$ for all $i>0$ and all $P\in \Pic^0(A)$.

     $\mathcal{F}$ is said to be \emph{continuously globally generated} if for any non-empty open subset $U$ of $\Pic^0(A)$, the evaluation map \[\operatorname{ev}_U:  \bigoplus_{Q\in U} H^0(A, \mathcal{F}\otimes Q)\otimes Q^{-1} \to \mathcal{F}\] is surjective.

\end{definition}
\begin{lemma}[{\cite[Lemma 2.2]{ChenHacon2007}}]\label{lem:IT-nonvanishing}
    Let $\mathcal{F}$ be a coherent IT$_0$ sheaf on an abelian variety $A$. If $\mathcal{F}\neq 0$, then $H^0(A,\mathcal{F}\otimes P)\neq 0$ for all $P\in \operatorname{Pic}^0(A)$.
\end{lemma}
The following result provides a criterion for pluricanonical pushforwards under Albanese morphism to satisfy IT$_0$.
\begin{theorem}[{\cite[Corollary 12.2]{LombardiPopaSchnell2020}}]\label{thm:LPS-IT}
   Let $X$ be a smooth projective variety with Iitaka fibration $I: X\to Z$ and Albanese map $f: X\to A$. If $q(X)=q(Z)$, and for some $m\geq 2$ we have $f_\ast \omega_X^{\otimes m}\neq 0$, then $f_\ast \omega_X^{\otimes m}$ satisfies IT$_0$.
\end{theorem}
\begin{definition}[{\cite[Lemma 2.3, Definition 2.4]{Jiang2021}}]
    Let $\mathcal{F}$ be a coherent sheaf on an abelian variety $A$ such that $h^0(A, \mathcal{F}\otimes Q)\neq 0$ for general $Q\in \Pic^0(A)$. There exists a unique continuously globally generated subsheaf $\mathcal{F}_c \subseteq \mathcal{F}$ such that $h^0(A, \mathcal{F}_c\otimes Q)=h^0(A, \mathcal{F}\otimes Q)$ for general $Q\in \Pic^0(A)$.

    We call $\mathcal{F}_c$ the \emph{maximal continuously globally generated subsheaf} of $\mathcal{F}$.
\end{definition}
Recall that a morphism $f: X\to A$ from a smooth projective variety to an abelian variety is called \emph{primitive} if the induced pull-back $f^\ast: \Pic^0(A) \to \Pic^0(X)$ is injective or equivalently $f$ does not factor through any non-trivial \'{e}tale cover of $A$.
\begin{definition}[{\cite[Definition 2.9]{Jiang2021}}]\label{def-eventual}
    Let $f: X\to A$ be a primitive morphism from a smooth projective variety to an abelian variety. Let $L$ be a line bundle on $X$ such that the image of the natural map $f^\ast (f_\ast L)_c\to L$ is non-zero. We denote by 
    \begin{center}
     \begin{tikzcd}
        X \arrow["\varphi_L",dashed]{r} \arrow["f"']{rrd} &Z_L \arrow[hook]{r} \arrow{rd} &\mathbb{P}_A((f_\ast L)_c)\arrow{d}\\
        &&A
    \end{tikzcd}   
    \end{center}  
    the relative evaluation map, where $Z_L$ is the image of $\varphi_L$.
    We call $\varphi_L$ the  \emph{eventual map} of $L$ (with respect to $f$).
\end{definition}

\subsection{Pluricanonical systems on surfaces of Kodaira dimension one}
We show some results about volumes and plurigenera of surfaces of Kodaira dimension one. 
\begin{theorem}[{\cite[Theorem 1.7]{ChenHanLiu2025}}]\label{thm:Ivol 1/6} 
    Let $S$ be a smooth projective surface with $\kappa(S)=1$. Then $\Iv(S)\geq \frac{1}{6}$, and this bound is optimal.
\end{theorem}
\begin{lemma}\label{lem:1/24}
    Let $S$ be a smooth projective surface with $\kappa(S)=1$. Then there exists an integer $m\in \{2,3,4\}$ such that $P_m>0$ and $\frac{\Iv(S)}{mP_m}\geq \frac{1}{24}$. This bound is sharp, as shown by the example in \cite[Remark 3.4]{CataneseLi2019}.
\end{lemma}
\begin{proof}
    We can assume that $S$ is minimal with elliptic fibration $f: S\to B$. Let $g:=g(B)$ and $\chi:=\chi(\OO_S)\geq 0$. The canonical bundle formula gives \[K_S=f^\ast \left(K_B+\Delta+\sum_{i=1}^{r}\left(1-\frac{1}{n_i}\right)p_i \right),\]
    where $\Delta$ is an integral divisor on $B$ of degree $\chi$, the points $p_1,\ldots,p_r$ support the multiple fibers, and $n_i\geq 2$ is the multiplicity of the fiber over $p_i$. Then \[P_m=h^0\left(B,mK_B+m\Delta+\sum_{i=1}^{r} \left\lfloor m\left(1-\frac{1}{n_i}\right) \right\rfloor p_i\right),\]
    and \[\Iv(S)=2g-2+\chi+\sum_{i=1}^{r}\left(1-\frac{1}{n_i}\right).\]
    By the estimate $h^0(B,L)\leq \max\{0, 1+\operatorname{deg}L\}$ for any line bundle $L$, \[P_m\leq \max\{0, 1+\operatorname{deg}\left(mK_B+m\Delta+\sum_{i=1}^{r}\left\lfloor m\left(1-\frac{1}{n_i}\right)\right\rfloor  \right)\}\leq 1+m\Iv(S).\]
    If $P_2>0$, using Theorem \ref{thm:Ivol 1/6}, we have \[\frac{\Iv(S)}{2P_2}\geq \frac{\Iv(S)}{2(1+2\Iv(S))}\geq \frac{1}{16}.\]
    Now suppose that $P_2=0$. The Riemann-Roch theorem gives
    \begin{equation}\label{eq:RiemannRoch in 1/24}
        0=P_2\geq\operatorname{deg}\left(2K_B+2\Delta+\sum_{i=1}^{r}\left\lfloor2-\frac{2}{n_i}\right\rfloor p_i\right)+1-g=3g-3+2\chi+r.
    \end{equation}
    The condition $\Iv(S)>0$ implies that
    \begin{equation}\label{eq: Ivol in 1/24}
        2g-2+\chi +r\geq 2g-2+\chi+\sum_{i=1}^{r}\left(1-\frac{1}{n_i}\right)>0
    \end{equation}
    Combining \eqref{eq:RiemannRoch in 1/24} and \eqref{eq: Ivol in 1/24}, one obtains \[0>g-1+\chi,\] so $g=\chi=0$ and $r=3$. Thus, $B\cong \Pp^1$ and 
    \begin{equation}\label{eq:plurigenera-formula}
        P_m=\max\left\{0,1-2m+\sum_{i=1}^{3} \left\lfloor m\left(1-\frac{1}{n_i}\right)\right\rfloor\right\}.
    \end{equation}
    If $n_i\geq 3$ for all $i$, then $P_3=1$ and \[\frac{\Iv(S)}{3P_3}=\frac{\Iv(S)}{3}\geq \frac{1}{18}.\] Now assume that $2=n_1\leq n_2\leq n_3$. If $n_2\leq 3$, $\Iv(S)=2g-2+\chi+\sum_{i=1}^{3}(1-\frac{1}{n_i})<\frac{1}{6}$, contradicting Theorem \ref{thm:Ivol 1/6}. If $n_2\geq 4$,  Equation \eqref{eq:plurigenera-formula} gives $P_4=1$. Thus, \[\frac{\Iv(S)}{4P_4}=\frac{\Iv(S)}{4}\geq \frac{1}{24}.\]
\end{proof}

\begin{lemma}\label{lem: surface-4-6-10}
    Let $S$ be a smooth projective surface with $\kappa(S)=1$. Then the following hold.
    \begin{enumerate}
        \item $P_4>0$, $P_6>0$ and $P_m>0$ for $m\geq 8$.
        \item $\lvert 10K_S \rvert$ defines the Iitaka fibration.
        \item If $P_5=0$, then $P_3,P_4,P_8>0$ and $\lvert 11K_S \rvert$ defines the Iitaka fibration.
        \item If $P_7=0$, then $P_5>0$ and $\lvert 11K_S \rvert$ defines the Iitaka fibration.
    \end{enumerate}
\end{lemma}
\begin{proof}
    We may assume that $S$ is minimal with the Iitaka fibration $f: S\to B$. We use the notation and formulas from the proof of Lemma \ref{lem:1/24} and use the same notation. If $g\geq 1$, the Riemann-Roch gives $P_m>0$ for $m\geq 2$. Indeed, when $g=1$ and $\chi=0$, $\Iv(S)>0$ implies that $r\geq 1$ so the Riemann-Roch still applies. Moreover,
 \[ \operatorname{deg}\left(10K_B+10\Delta+\sum_{i=1}^{r} \left\lfloor 10\left(1-\frac{1}{n_i}\right)\right\rfloor\cdot p_i \right)\geq 2g+1,\] so the corresponding divisor is very ample. Thus, $\lvert 10K_S \rvert$ defines the Iitaka fibration.

    Now assume that $g=0$. Then $\lvert mK_S \rvert$ defines the Iitaka fibration if and only if $P_m\geq 2$. By the plurigenus formula
    \begin{equation}
        P_m=\max\left\{0,1+m(\chi-2)+\sum_{i=1}^{r} \left\lfloor m\left(1-\frac{1}{n_i}\right)\right\rfloor\right\}.
    \end{equation} 
    \begin{enumerate}
        \item If $\chi \geq 3$, the formula gives $P_m\geq 2$ for $m\geq 1$.
        \item If $\chi=2$, $\Iv(S)>0$ gives $r\geq 1$ so $P_m\geq 2$ for $m\geq 2$.
        \item If $\chi =1$, $\Iv(S)>0$ implies $r\geq 2$. Moreover, if $r=2$, at most one $n_i$ is equal to $2$. Hence, $P_m>0$ for $m\geq 3$ and $P_{10}\geq 2$.
        \item  If $\chi =0$, the multiplicity analysis in \cite[Section 4, Case (5)]{CataneseLi2019} shows that, after arranging the multiplicities in non-decreasing order, they dominate one of the tuples listed below componentwise. Here additional multiple fibers are allowed. We can get the following lower bounds for the plurigenera:
    \begin{center}
        \begin{table}[ht]
\centering
\begin{tabular}{c|ccccc}
\toprule
Multiplicity tuple
& $P_3$ & $P_5$ & $P_7$ & $P_{10}$ & $P_{11}$ \\
\midrule
$(2,2,2,2,2)$ & $0$ & $1$ & $2$ & $6$ & $4$ \\
$(2,2,3,3)$   & $1$ & $1$ & $1$ & $3$ & $3$ \\
$(4,4,4)$     & $1$ & $0$ & $2$ & $2$ & $3$ \\
$(3,6,6)$     & $1$ & $2$ & $1$ & $3$ & $4$ \\
$(3,4,12)$    & $1$ & $1$ & $2$ & $3$ & $4$ \\
$(2,5,10)$    & $0$ & $1$ & $1$ & $3$ & $1$ \\
$(2,6,6)$     & $0$ & $1$ & $0$ & $2$ & $2$ \\
\bottomrule
\end{tabular}
\caption{Lower bounds for the plurigenera.}
\label{tab:plurigenera}
\end{table}
    \end{center}
Similarly, one obtains that $P_4, P_6, P_8, P_9>0$ in every row. The conclusion follows directly from the table.
    \end{enumerate}

\end{proof}
\subsection{Generically isotrivial threefolds}\label{subsec: geneiso}
Let $X$ be a
smooth projective threefold with
$\kappa(X)=2$.
We say $X$ is \emph{generically isotrivial}, if its general fibers of the Iitaka fibration are mutually isomorphic.

By \cite[Theorem 1.6]{HaoSchreieder2021} or Theorem \ref{thm:hao-schreieder} below, $X$ is generically isotrivial, if and only if
$X\sim_{\text{bir}} (T\times E) /G$, where 
\begin{itemize}
    \item[(1)] $T$ is a projective surface with canonical singularities and ample $K_T$, $E$ is an elliptic curve, and $ (T\times E) /G$ has canonical singularities;
    \item[(2)] $G$ acts diagonally and faithfully on each factor;
    \item[(3)] The natural quotient map $q:T\times E\to (T\times E) /G$ is \'{e}tale in codimension one.
\end{itemize}
Denote by $S$ the smooth minimal model of resolution of $T$.
Then $\Aut(S)=\Aut(T)$ by \cite[Section 2.5, p.183]{Sym4fold08} and $G$ acts faithfully on $S$.

\subsection{Numerical criterion for separation of adjoint system}
We introduce some numerical criteria for separating two distinct closed points by adjoint linear system.
These criteria mainly contribute to arguments in Section \ref{section: effIitgeneisotrivial}.

\begin{lemma}[]\label{lem: criterionforsepa}
Suppose that $Y$ is a smooth projective surface.
Let $L$ be a big and nef $\mathbb{Q}$-divisor on $Y$.
Assume there exists $\beta \in \mathbb{R}$ with $\beta > 8$
such that
\begin{enumerate}
    \item $L^2 \geq  \beta$;
    \item for every very general point $P \in Y$ and every
    irreducible curve $C \subset Y$ containing $P$, one has
    \[
        L \cdot C
        >
        \frac{4}{1+\sqrt{1-\frac{8}{\beta}}}.
    \]
\end{enumerate}
Then two distinct very general points of $Y$ are separated by
\[
    \left| K_Y + \lceil L \rceil \right|.
\]
In particular, this linear system defines a birational map.
\end{lemma}
\begin{proof}
Since $L$ is $\Q$-Cartier divisor, there is a positive number $M$ such that $ML$ is Cartier and $(L\cdot C) \in \frac{1}{M}\N$ belongs to a discrete subset for all irreducible $C$.
So there is a positive real number $0<\varepsilon\ll 1$ such that 
\[
L^2 >\beta-\varepsilon>8  , \qquad L \cdot C
        > \frac{4}{1+\sqrt{1-\frac{8}{\beta-\varepsilon}}} 
\]
Replacing $\beta$ by $\beta-\varepsilon$, one can apply \cite[Theorem~2.8]{CCCJ21} on $L, C$ and obtain the result.
\end{proof}

\subsection{Inequalities on fibrations}
We recall two well-known inequalities for surface fibrations.
\begin{lemma}\label{lem: slopeineqs}
Let $f:S\to B$ be a fibration from a smooth surface of general type to a smooth curve with general fiber $F$.
Denote $\chi:=\chi(\OO_S)$, $V:=\vol(S) >0$, $g:=g(F)\geq 2$ and $b:=g(B)$.
Then 
\begin{itemize}
    \item[(1)] If $b=1$, then $V\geq \chi \cdot\frac{4(g-1)}{g}$;
    \item[(2)] If $b\geq 2$, then $V\geq 8(g-1)(b-1)$.
\end{itemize}
\end{lemma}
\begin{proof}
Since $b>0$, one can contract $f$-vertical $(-1)$-curves and assume $S$ is minimal.

Assertion (1) follows by Xiao's slope inequality \cite[Theorem 2]{Xiao87}, via taking $b=g(B)=1$, $\chi(\OO_B)=0$, $\omega_B\sim 0$, and  $\lambda(f)=\frac{K_S^2}{\deg(f_*\omega_{S})}=\frac{V}{\chi}$ here.

Assertion (2) follows by \cite[Theorem 2.2]{Catanese2000}.
\end{proof}

\subsection{Abelian automorphism groups of smooth minimal surfaces}
We introduce some notations mostly derived from \cite{lu2025sharp}.
Let $S$ be a smooth minimal surface of general type and $G\leq \Aut(S)$ be a finite abelian subgroup.
Denote the geometric genus, volume, irregularity and algebraic Euler characteristic of $S$ by $p_g,V,q,\chi$ respectively.
\begin{definition}
A pencil $\Lambda$ on $S$ is called \emph{$G$-invariant}, if $g.\Lambda=\Lambda$ for any $g\in G$.
A fibration $f:X\to C$ on a smooth model $X$ of $S$ is called \emph{$G$-invariant}, if any fiber is sent to a fiber by any $g\in G$.
We say $\Lambda$ is \emph{rational} (resp. \emph{irrational}), if the associated fibration $f:X\to C$ satisfies $g(C)=0$ (resp. $g(C)\geq 1$).

Let $H_n:=H^0(S,nK_S)$ be the $n$-canonical system and $v\in H_n$ be an element.
We say $v$ is a \emph{semi-invariant vector}, if for any $g\in G$, there is $\lambda_g\in \C^*$ with
$g.v=\lambda_g v$. 

If $v_1,v_2\in H_n$ are linearly independent semi-invariant vectors, then $\Lambda_n:=\operatorname{span}_\C\{v_1,v_2  \}$
generates a $G$-invariant pencil.
We say a fibration $f_n:X\to C$ is a \emph{$G$-invariant $n$-canonical pencil}, if $f_n$ is induced by Stein factorization of the resolution of the map defined by $\Lambda_n$.

We also define the \emph{canonical degree} $\lambda$ of pencil $\Lambda$ as $\lambda:=\frac{K_S\cdot M}{K_S^2}$, where $M$ is the moving part of $\Lambda$.   
\end{definition}

\subsubsection{Construct pluricanonical $G$-invariant pencils}
In the following, we construct $G$-invariant pluricanonical pencils by two linearly independent semi-invariant vectors. 
\begin{lemma}\label{lem: createGinvpencil}
Let $S$ be a smooth minimal surface of general type.
Let $G\leq \Aut(S)$ be a finite abelian subgroup.
Then there is a resolution $S'$ of $S$ satisfying the following conditions.
\begin{itemize}
    \item[(1)] If $p_g\geq 2$, then there is a $G$-invariant canonical pencil $f:S' \to B$ with canonical degree $\lambda \leq 1$;
    \item[(2)] If $p_g\leq 1$, then there is a $G$-invariant bicanonical pencil $f:S' \to B$ with canonical degree $\lambda \leq 2$.
\end{itemize}
\end{lemma}
\begin{proof}
    We consider the $n$-th pluricanonical system $V_n:=H^0(S,nK_S)$ for some $n\in \N_{\geq 1}$ and $h^0(S,nK_S) \geq 2$.
    Then there is a natural complex representation
    \[
    \rho: G \to \GL(V_n), \qquad g\mapsto(g^*: s \mapsto g^*s)
    \]
    induced by automorphisms on $S$.
    
    Since $G$ is abelian, every irreducible sub-representation has dimension one.
    So there are two semi-invariant vectors $s_1,s_2\in V_n$ such that 
    \begin{equation}\label{eq: s1s2semiinv}
    g^*s_1=c_1(g)s_1, g^*s_2=c_2(g)s_2, \qquad c_1(g),c_2(g)\in \C^*   
    \end{equation}
    for any $g\in G$.
    After choosing a resolution $\pi:S'\to S$ of $S$, there is a morphism $h:S'\to \Pp^1$ defined by $n$-canonical global sections $s_1,s_2$.
    Taking the Stein factorization 
    $S'\xrightarrow[]{f} B \xrightarrow[]{g} \Pp^1$ of $h$, $S'\xrightarrow[]{f} B$ is a $G$-invariant $n$-canonical fibration induced by a pencil with general fiber $F$.
    Moreover, the canonical degree \[\lambda=\frac{(K_{S}\cdot M_n)}{(K_{S}\cdot K_S)}\leq \frac{(K_{S}\cdot nK_S)}{(K_{S}\cdot K_S)}=n,\]
    where $M_n$ is the movable part of that pencil.
    
    If $p_g\geq 2$, then one can define $n=1$ and obtain (1).
    If $p_g\leq 1$, then $h^0(S,2K_S)\geq 2$ by Riemann-Roch formula and we obtain (2) by taking $n=2$.
\end{proof}

\subsubsection{Abelian automorphism groups with a rational $G$-invariant pencil}
The following Lemmas help us to bound the order of abelian automorphism group $G$ with a rational $G$-invariant pencil in Proposition \ref{prop: ratpencilcase}.

\begin{lemma} \label{lem: resolintnumcal}
Let $S$ be a smooth minimal projective surface of general type.
Suppose that there is a pencil $\Lambda$ on $S$.
Let $\pi:S'\to S$ be the resolution of base points of $\Lambda$.
Let $f:S' \to C$ be the fibration induced by $\Lambda$ with general fiber $F$.
By definition, the canonical degree $\lambda\geq \frac{K_S\cdot \hat{F}}{K_S^2}$, where $\hat{F}:=\pi(F)$.

By Castelnuovo's theorem, one can decompose $\pi:S'\to S$ as $r$ successive blow-ups at smooth points $P_i\in S_i$:
$S':=S_r \xrightarrow[]{\pi_r} S_{r-1} \xrightarrow[]{\pi_{r-1}} \dots \to S_0:=S$.
Denote $a_i:=\operatorname{mult}_{P_i}(F_i)$, where $F_i$ is image of $F$ on $S_i$.

Then 
\begin{itemize}
     \item[(1)] $2g(F)-2=K_{S'}\cdot F=\sum_{0\leq i\leq r-1} a_i +K_{S}\cdot \hat{F}\leq\sum_{0\leq i\leq r-1} a_i +\lambda V$;
    \item[(2)] $\hat{F}^2=\sum_{0\leq i\leq r-1}a_i^2$;
    \item[(3)] $g(F)=p_a(\hat{F})-\sum_{0\leq i\leq r-1} \frac{a_i(a_i-1)}{2}$;
    \item[(4)] $r\leq \sum_{0\leq i\leq r-1}a_i\leq \sum_{0\leq i\leq r-1}a_i^2\leq \lambda^2 V$.
\end{itemize}
\end{lemma}
\begin{proof}
    By assumptions, for $0\leq i\leq r-1$, we might write $F_{i+1}=\pi_{i+1}^*F_i-a_iE_{i+1}$ and $K_{S_{i+1}}=\pi_{i+1}^*K_{S_i}+E_{i+1}$,
where $E_{i+1}=\operatorname{Exc}(\pi_{i+1})$ and $E_{i+1}^2=-1$.

For assertions $(1)\sim (3)$, it suffices to prove that, for each step $\pi_{i+1}:S_{i+1}\to S_{i}$ of the resolution,
\[
\begin{cases}    
    K_{S_{i+1}}\cdot F_{i+1}= K_{S_{i}}\cdot F_{i}+a_i;                     \\
    F_{i+1}^2=F_{i}^2-a_{i}^2; \\
    p_a(F_{i+1})=p_a(F_{i})-\frac{a_{i}(a_{i}-1)}{2} \\    
\end{cases}
\]
holds respectively.
In fact, we have 
\[
\begin{cases}    
    K_{S_{i+1}}\cdot F_{i+1}=  (\pi_{i+1}^*K_{S_i}+E_{i+1})    \cdot(\pi_{i+1}^*F_i-a_iE_{i+1})=K_{S_i}\cdot F_i +a_i;                                \\
    F_{i+1}^2=(\pi_{i+1}^*F_i-a_iE_{i+1})^2=(\pi_{i+1}^*F_i)^2+a_i^2E_{i+1}^2=F_i^2-a_i^2;  \\
    2p_a(F_{i+1})-2=((K_{S_{i+1}}+F_{i+1})\cdot F_{i+1})=(K_{S_{i}}\cdot F_i+a_i)  +(F_i^2-a_i^2)=2p_a(F_i)-2-(a_i^2-a_i).   
\end{cases}
\]
Assertion (4) holds as 
\[
\sum_{0\leq i\leq r-1}a_i\leq \sum_{0\leq i\leq r-1}a_i^2= \hat{F}^2 \leq \frac{(K_S\cdot \hat{F})^2}{K_S^2} \leq \lambda^2 V,
\]
where the equality $ \sum_{0\leq i\leq r-1}a_i^2= \hat{F}^2$ follows from (2).
\end{proof}

\begin{lemma}\label{lem: calcu}
Denote by $S$ the minimal resolution of $T$.
Let $G \leq \Aut(T)=\Aut(S)$ be a finite abelian subgroup.
Suppose that there is a rational $G$-invariant pencil $f:S'\to \Pp^1$ with general fiber $F$ where $g:=g(F) \geq 2$
and canonical degree $\lambda$.

Then \[|G| \leq \max \left\{ \left( 2+\frac{4}{g-1}\right)\cdot (3\lambda^2+4\lambda +1)V, \left( 2+\frac{1}{g} \right)\cdot (\chi_{\text{top}}(S)+(3\lambda^2+2\lambda)V )   \right\}. \]
\end{lemma}

\begin{proof}
By \cite[Theorem 3.1]{lu2025sharp}, we have 
\begin{equation}\label{eq: ineqthmLX}
|G| \leq  \max \left\{ \left( 2+\frac{4}{g-1}\right)K_f^2 , \left( 2+\frac{1}{g} \right)e_f   \right\}.
\end{equation}

Moreover, we can bound $K_f^2$ and $e_f$ as follows
\begin{equation}\label{eq: ineqGKf}
\begin{aligned}
K_f^2 &=K_{S'}^2-2(2g-2)(2g(\Pp^1)-2) \\
      &=V-r+8(g-1)  &\text{where } r\leq \rho(S'/S)  \\  
      &\leq(1+4\lambda)V+\sum_{0\leq i\leq r-1} (4a_i-1)  &\text{by Lemma \ref{lem: resolintnumcal} } (1)             \\
      &\leq (1+4\lambda)V+\sum_{0\leq i\leq r-1} 3a_i^2 &\text{by } a_i\geq 1 \\
      &\leq (1+4\lambda+3\lambda^2)V  &\text{by Lemma \ref{lem: resolintnumcal} } (4),
\end{aligned}    
\end{equation}
and 
\begin{equation}\label{eq: ineqGef}
\begin{aligned}
e_f &=\chi_{\text{top}}(S')-\chi_{\text{top}}(F) \chi_{\text{top}}(\Pp^1)\\
      &=\chi_{\text{top}}(S)+r+4(g-1)  \\
      &\leq \chi_{\text{top}}(S)+\lambda^2V+2\lambda(1+\lambda)V     &\text{by Lemma \ref{lem: resolintnumcal} } (1),(4)           \\
      &=\chi_{\text{top}}(S)+(3\lambda^2+2\lambda)V. 
\end{aligned}    
\end{equation}
Hence we complete the proof by combining inequalities \eqref{eq: ineqthmLX}, \eqref{eq: ineqGKf} and \eqref{eq: ineqGef} together.
\end{proof}
\subsubsection{Irrational $G$-invariant pluricanonical pencils}
The following lemmas help us to bound the order of abelian automorphism group $G$ with an irrational $G$-invariant pencil in Proposition \ref{prop: irratpencilcase}.
\begin{lemma}\label{lem: Qleqk}
Let $S$ be a smooth minimal surface of general type with $\vol(S)=V$ and a smooth model $S'$.
Let $G\leq \Aut(S)$ be a finite abelian subgroup.
Suppose that there is an $n$-th canonical fibration $h:S'\to \Pp^1$ defined by two semi-invariant vectors $s_1,s_2\in H^0(S,nK_S)$.
Denote by $h:S'\xrightarrow[]{f} B \xrightarrow[]{g} \Pp^1$ the Stein factorization of a $G$-invariant pencil with $\deg g=k$.
Denote $Q:=\im(G \to \Aut(B))$.
Assume that $B$ is an elliptic curve.

If $Q$ acts on $B$ by translations, then $|Q|\leq k$.
\end{lemma}
\begin{proof}
Denote $u:=\frac{s_1}{s_2}$ and $\operatorname{div}(u):=D_1-D_2$, where $D_1,D_2\geq 0$ without common components.
Then the divisor $D_2=g^{*}(\infty)>0$.
Since $s_1,s_2$ are semi-invariant and $Q$ is image of $G$, $\operatorname{div}(g.u)=\operatorname{div}(u)$ and $D_2$ is $Q$-invariant.
Take a closed point $P\in D_2$.
As $Q$ acts on $B$ by translations, $q_1(P)\neq q_2(P) \in \operatorname{Supp}(D_2)= \operatorname{Supp}(g^{-1}(\infty))$ for any $q_1\neq q_2\in Q$.
So $|Q|\leq |\operatorname{Supp}(g^{-1}(\infty))|$.
However,
we have $k=\deg g\geq |\operatorname{Supp}(g^{-1}(\infty))|$ by definition.
Hence $|Q|\leq |\operatorname{Supp}(g^{-1}(\infty))|\leq k$.

\end{proof}
\begin{lemma}\label{lem: bdGwhendimalb=1andqgeq2}
Let $S$ be a smooth minimal surface of general type with volume $V$.
Let $f:S\to B$ be a $G$-invariant fibration to smooth projective curve $B$ with $b:=g(B)\geq 2$.
Denote by $g\geq 2$ the genus of general fiber of $f$.
Then $|G|\leq 18V$.
\end{lemma}
\begin{proof}
Let $K$ be the kernel of $G\to\Aut(B)$ and let $Q$ be its image. 
Since $K$ is finite, the action of $K$ on general $F$ is faithful so that $K\leq \Aut (F)$ is a subgroup.
Since $K$ and $Q$ are both abelian, one has $|G|\leq (4g+4)(4b+4)$ by Lemma \ref{lem: curveautos}.
By Lemma \ref{lem: slopeineqs} (2), we have
\begin{equation*}
|G|\leq  2\,\frac{g+1}{g-1}\frac{b+1}{b-1}V\leq 18V.
\end{equation*}
\end{proof}

\subsection{Abelian automorphism groups of irrational curves}
The following Lemmas bound the order of abelian automorphism groups of irrational curves.
\begin{lemma}\label{lem: bdQwithlinearpart}
Let $Q$ be a finite abelian automorphism  group of an elliptic curve $B$. 
Then either $Q$ consists of translations, or $|Q|\leq 9$.
\end{lemma}

\begin{proof}
Let $N$ be the translation subgroup of $Q$, and let the linear part of $Q$ have order $m\in\{1,2,3,4,6\}$. 
If $m=1$, then $Q$ consists of translations.
If $m>1$, write $Q:=N\rtimes \Z_m \leq \Aut(B)$ where $m>1$ and $\zeta_m$ is $m$-th root of the unity, as a generator of linear part $\Z_m$.
Since $Q$ is abelian, $aba^{-1}b^{-1}=1$ for any $a,b\in Q$.

However, if we write $a,b$ as automorphisms on $B$:
\[
a: z\mapsto w_az+t_a, \qquad b:z\mapsto w_bz+t_b, \qquad z\in B,
\]
then $aba^{-1}b^{-1}$ is the transition
\[
z \overset{b^{-1}}{\mapsto} \frac{z-t_b}{w_b} \overset{a^{-1}}{\mapsto}
\frac{z-(w_bt_a+t_b)}{w_aw_b} \overset{b}{\mapsto} \frac{z-(w_bt_a+t_b)+w_at_b}{w_a} \overset{a}{\mapsto} z+(w_a-1)t_b-(w_b-1)t_a.
\]
Since $aba^{-1}b^{-1}=1$, $(w_a-1)t_b=(w_b-1)t_a$ for every $a,b\in Q$.
Take $a=(0,\zeta_m)\in N\rtimes \Z_m= Q$ so that $t_a=0,w_a=\zeta_m$ and run $b$ through $N$.
Then we get $t_b\in \ker(\zeta_m-1)$ for any $b\in N$.

Moreover, the isogeny $(\zeta_m-1):B\to B$ has degree $|\zeta_m-1|^2$.
Hence $|Q|\leq |\ker(\zeta_m-1)|\cdot |\Z_m|=m|\zeta_m-1|^2$ is upper bounded by $8,9,8,6$ for $m=2,3,4,6$
respectively.
\end{proof}

\begin{lemma}[See {\cite[Theorem 1]{nakajima1987abelian}}]\label{lem: curveautos}
Let $C$ be a smooth curve with $g:=g(C)\geq 2$.
Let $G\leq \Aut(C)$ be an abelian subgroup.
Then $|G|\leq 4g+4$.
\end{lemma}

\subsection{Abelian group action on $(1,2)$-surfaces}
The following lemma contributes to dealing with the extremal case in Theorem \ref{mainthm: eff-iitaka} and Theorem \ref{mainthm: iitaka-vol}.
\begin{lemma}\label{lem: whenT12G100}
Let $T$ be a normal projective surface with canonical singularities and ample $K_T$.
Let $G\leq \Aut(T)$ be an abelian automorphism group.

Assume $(\vol(T),p_g(T))=(1,2)$.
Then 
\begin{itemize}
\item[(1)] there are weighted homogeneous coordinates $x,y,z,w$ with weight $1,1,2,5$
respectively, such that
\[ T\cong \{w^2=F_{10}(x,y,z)\}\subset\Pp(1,1,2,5),\qquad
 R(T,K_T)=\C[x,y,z,w]/(w^2-F_{10})
\]
and $G$ acts on $x,y,z,w$ by scalars;
\item[(2)] $|G|\leq 100$. If $|G|=100$,  then $F_{10}(x,y,z)=x^{10}+y^{10}+z^5$, after possibly rescaling coordinates.
Moreover, the
$G$-action on $T$ can be characterized as  
\[
 g.([x:y:z:w])=
([\zeta_{10}^sx:\zeta_{10}^ty:z:w])
\]
for any $g=(s,t)\in \Z_{10} \times \Z_{10},[x:y:z:w]\in T$.
\end{itemize}
\end{lemma}
\begin{proof}
For (1), by \cite[Chapter VII, \S~7]{BHPV2004} and $(\vol(T),p_g(T))=(1,2)$, 
\[
 T\cong \{w^2=F_{10}(x,y,z)\}\subset\Pp(1,1,2,5),\qquad
 R(T,K_T)=\C[x,y,z,w]/(w^2-F_{10}(x,y,z)), 
\] 
where weighted homogeneous coordinates $x,y,z,w$ have weight $1,1,2,5$
respectively, and monomial $z^5$ occurs in $F_{10}(x,y,z)$.
By Lemma \ref{lem: PS17} below, we can assume that $G$ acts on $x,y,z,w$ by scalars. 

Then we prove (2). 
We use the normalized coordinates from (1) so that $G$ acts on $x,y,z,w$ by scalars $\alpha_g,\beta_g,\gamma_g,\delta_g \in \C^*$ respectively. 
Since $F_{10}$ is semi-invariant, $\gamma_g^5=\delta_g^2$ by comparing $G$-actions on monomials $z^5,w^2$. 
Set $\lambda_g:=\frac{\delta_g}{\gamma_g^2}$ and then $\gamma_g=\lambda_g^2,\delta_g=\lambda_g^5$.
So $G$ acts on $T$ as
\[
G\hookrightarrow(\mathbb C^*)^2,\qquad
g\cdot[x:y:z:w]=[\alpha_gx:\beta_gy:\lambda_g^2z:\lambda_g^5w]=[{\lambda_g}^{-1}\alpha_gx:\lambda_g^{-1}\beta_gy:z:w],
\]
Replacing $\alpha_g,\beta_g$ by ${\lambda_g}^{-1}\alpha_g,{\lambda_g}^{-1}\beta_g$ respectively, we might assume $\lambda_g=1$.

We claim that there are two distinct vectors $(a,b),(c,d) \in \Gamma$ spanning a rank-two lattice $L_0$, where $\Gamma:=\{(a,b)\in\mathbb Z^2|
             x^ay^bz^c \text{ is a monomial in } F_{10} \text{ for some } c\geq 0\}$. 
Indeed, if not, denote $u:=(a,b)\in \Z^2$.
Consider $(w^2-F_{10})$-invariant actions
\[
\lambda:[x:y:z:w] \mapsto [\lambda^bx:\lambda^{-a}y:z:w], \qquad \lambda \in \C^*, 
\]
which generates an infinite subgroup of $\Aut(S)$.
This contradicts the boundedness of $\Aut(S)$ for general type $S$.

For those two linearly independent vectors
$(a,b), (c,d)\in L_0\cap \Gamma $,
one has
\[
G\subseteq
H:=\{(\alpha,\beta)\in(\mathbb C^*)^2:
\alpha^a\beta^b=\alpha^c\beta^d=1\},
\]
and $|H|=|ad-bc|$.
Since $a+b\leq 10$ and $c+d\leq10$, we obtain
\[
|G|
\leq |ad-bc|
\leq ad+bc
\leq(a+b)(c+d)
\leq100,
\]
and equality holds if and only if
$a+b=c+d=10, ac=bd=0$ and
$(a,b)=(10,0),(c,d)=(0,10)$.
Hence $|G|=100$ implies $G=\Z_{10}\times\Z_{10}$, and
 $g:=(s,t)\in G\cong (\mathbb Z/10\mathbb Z)^2$ acts on $T$ as
\[g\cdot[x:y:z:w]=[\zeta_{10}^{s}x:\zeta_{10}^{t}y:z:w].\]

Finally, since $F_{10}$ is invariant under $\Z_{10}\times \Z_{10}$-action,
$F_{10}(x,y,z)=Ax^{10}+By^{10}+Cz^5$ and $ABC\neq 0$ by \cite[Chapter VII, \S 7]{BHPV2004}.
Then we obtain $F_{10}(x,y,z)=x^{10}+y^{10}+z^5$ after possibly rescaling the coordinates.
\end{proof}

The following lemma is derived from \cite[Appendix A]{PS17}.
\begin{lemma}\label{lem: PS17}
Let $T$ be a normal projective variety with canonical
singularities and ample $K_T$.
Denote by $R(T,K_T)$ the canonical ring of $T$.
Let
$G\leq\operatorname{Aut}(T)$ be a finite abelian subgroup.
Suppose that 
\[
R(T,K_T)\cong
\mathbb C[x_0,\ldots,x_N]/(F_1,\ldots,F_c)
\]
where the $x_i$ are minimal homogeneous generators.

Then $G$ acts on $x_i$ by scalars after a graded coordinate change.
\end{lemma}

\begin{proof}
Denote by $R_m\subset R:=R(T,K_T)$ the subspace  of degree $m$ in a graded ring.
Let $R_+:=\sum_{m>0} R_m\leq R$ be a graded ideal of $R$.
Then $R_+^2$ is also a graded ideal and $(R_+^2)_m=\sum_{0<k<m} R_k R_{m-k}\leq R$ for each $m\in \N_{\geq 1}$.

Note that $R_m$ carries a natural graded $G$-action.
As in the proof of \cite[Lemma A.2.13]{PS17},
complete reducibility gives $G$-invariant complements $V_m$ such that
\[
R_m=(R_+^2)_m\oplus V_m.
\]
Since $G$ is finite abelian, choose an eigenbasis of
each $V_m$. 
The union of these bases forms a minimal homogeneous generating set of $R$ with the same degrees
as those of $x_i$.
Replacing $x_i$ by these coordinates,  $G$ acts on $x_i$ by scalars.
\end{proof}

\section{A structure theorem for $n$-folds of Kodaira dimension $n-1$}
We first recall a structure theorem for generically isotrivial Iitaka fibrations. This theorem does not require an irregularity assumption and will be used to study threefolds of Kodaira dimension two:
\begin{theorem}[{\cite[Theorem 1.6]{HaoSchreieder2021}}]\label{thm:hao-schreieder}
    Let $X$ be a minimal model of dimension $n$ and Kodaira dimension $n-1$. Then the following conditions are equivalent.
    \begin{enumerate}
        \item The Iitaka fibration $f: X\to S$ is generically isotrivial.
        \item There is an elliptic curve $E$, a normal projective variety $T$ with ample canonical class $K_T$, and a finite subgroup $G\subset \operatorname{Aut}(T)\times \operatorname{Aut}(E)$ with the following properties. The action of $G$ on $T\times E$ is diagonal and faithful on each factor, such that the quotient $Y:=(T\times E)/G$ has only canonical singularities. Moreover, there is a commutative diagram
        \begin{center}
            \begin{tikzcd}
                X \arrow["\phi", dashed]{r} \arrow["f"']{rd} &X^+ \arrow["\tau"]{r} &Y=(T\times E)/G \arrow["g"]{ld}\\
                &S\cong T/G,&
            \end{tikzcd}
        \end{center}
        where $f$ is the Iitaka fibration of $X$, $\phi$ is a composition of flops, $\tau$ is a terminalization, and $g$ is induced by the projection $T\times E\to T$. Moreover, there is a normal subgroup $G_0\subset G$ of index $\in\{1,2,3,4,6\}$ whose action on $T\times E$ is free.
    \end{enumerate}
\end{theorem}

The following theorem relates morphisms to abelian varieties to the Iitaka fibration.

\begin{theorem}[{\cite[Theorem 4.4]{ChenChurchHao2025}}]\label{thm: chen-church-chan}
    Let $X$ be a smooth projective variety of dimension $n$. Suppose there exists a morphism $f: X\to A$ to an abelian variety $A$ such that the general fiber of the Iitaka fibration is not contracted by $f$. Then there exists a $(\operatorname{dim}X-\kappa(X))$-dimensional quotient $q: A\to B$ and a birational map $X\dashrightarrow (B^\prime \times Z)/G_0$ making the diagram 
    \begin{center}
        \begin{tikzcd}
            X\arrow["f"]{r} \arrow[dashed]{d} &A \arrow["q"]{d}\\
            (B^\prime \times Z)/G_0 \arrow{r} &B
        \end{tikzcd}
    \end{center}
    commute. Here, $B^\prime \to B$ is an isogeny with kernel $G_0$ and $Z$ is a smooth variety of general type with a (not necessarily free) $G_0$-action. $G_0$ acts on $B^\prime \times Z$ diagonally.
\end{theorem}
Although our main results concern irregular varieties, part of the corollary holds without an irregularity assumption.
\begin{corollary}\label{cor: chargeneisotrivial3fold}
    Let $X$ be a minimal model of dimension $n$ and Kodaira dimension $n-1$. Let $f: X\to Z$ be its Iitaka fibration. 
    \begin{enumerate}
        \item If $f$ is generically isotrivial, then $\Iv(X)$ of $X$ is equal to \[\frac{\vol(T)}{\lvert G \rvert },
        \] where $T, E$ and $G$ are given in Theorem \ref{thm:hao-schreieder}. Since the action is faithful on each factor, $G$ embeds into both $\operatorname{Aut}(T)$ and $\operatorname{Aut}(E)$. In particular, \[G\cong (\mathbb{Z}/m\mathbb{Z}\times \mathbb{Z}/n\mathbb{Z})\rtimes \mathbb{Z}/\mu\mathbb{Z},\] where $n\vert m$ and $\mu\in\{1,2,3,4,6\}$.
        \item If $f$ is not generically isotrivial and $q(X)>0$, then the Albanese morphism of $X$ factors through the Iitaka fibration $f$, i.e. \[X\xrightarrow{f} Z \rightarrow \operatorname{Alb}(X).\]
        \item If $q(X)>0$ and the Albanese morphism of $X$ does not factor through the Iitaka fibration $f$, then $f$ must be generically isotrivial and the $G$ in Theorem \ref{thm:hao-schreieder} acts on $E$ by translations. In particular, $G\cong \mathbb{Z}/m\mathbb{Z}\times \mathbb{Z}/n\mathbb{Z}$ where $n\vert m$.
    \end{enumerate}
\end{corollary}
\begin{proof}
We first prove (3). Apply Theorem \ref{thm: chen-church-chan} to the Albanese morphism $a_X: X\to A:=\Alb(X)$, after passing to a smooth birational model if necessary. The proof of Theorem \ref{thm: chen-church-chan} (\cite[Theorem 4.4]{ChenChurchHao2025}) shows that the natural projection $(B^\prime \times Z)/G_0\to Z/G_0$ is birationally equivalent to the Iitaka fibration $f$. In particular, $f$ is generically isotrivial.
By Theorem \ref{thm:hao-schreieder}, $f$ is also birationally equivalent to $(T\times E)/G\to T/G$.
By \cite[Lemma 5.20]{KM98}, \[(T\times E)/G,\; (B^\prime\times Z)/G_0,\;T/G,\;Z/G_0\] are all klt.
Then \cite[Theorem 5.22]{KM98} implies they have rational singularities.
So \[q((B^\prime \times Z)/G_0)= q((T\times E)/G),\quad q(T/G)=q(Z/G_0),\] since the two total spaces are birational, as are the two bases.

Since any $G$ acts trivially on constant global section of structure sheaf, the K\"unneth formula gives \[q((B^\prime \times Z)/G_0)=q(B^\prime/G_0)+q(Z/G_0),\qquad q((T\times E)/G)=q(T/G)+q(E/G).\]
Comparing these equalities yields  $q(B'/G_0)=q(E/G)$. 
Since $G_0$ acts on the elliptic curve $B^\prime$ by translations, we have $q(E/G)=q(B'/G_0)=1$. 
So $G$ acts trivially on $H^0(E, \Omega_E^1)$ and therefore acts on $E$ by translations. This proves (3), and (2) follows immediately. Note that $Z\dashrightarrow \Alb(X)$ is indeed a morphism by \cite[Lemma 8.1]{Kawamata1985}.

To prove (1), consider the following commutative diagram:
\begin{center}
    \begin{tikzcd}
       T \times E \arrow[r,"\operatorname{pr}_1"] \arrow[d,"\alpha"] & T \arrow[d,"\delta"] \\
        Y\colon=(T\times E)/G \arrow[r,"\pi"] & T/G. 
    \end{tikzcd}
\end{center}
Notice that $\pi$ is the Iitaka fibration by Theorem \ref{thm:hao-schreieder} and $K_Y=\pi^\ast H$, where $H$ is an ample $\Q$-divisor on $T/G$. 
Since $\alpha$ is quasi-\'{e}tale by arguments in \cite[Theorem 1.6]{HaoSchreieder2021},
\[K_{T\times E}=\alpha^\ast K_Y=\alpha^\ast \pi^\ast H=\operatorname{pr}_1^\ast \delta^\ast H,\] 
and $\delta^\ast H$ is just $K_T$. 
Hence we obtain 
\[\operatorname{Ivol}(X)=H^{n-1}=\frac{1}{\lvert G\rvert}(\delta^\ast H)^{n-1}=\frac{\operatorname{vol}(T)}{\lvert G \rvert}.\]

\end{proof}

\section{The case where the Albanese map factors through the Iitaka fibration}
 Let $X$ be a smooth projective irregular threefold with $\kappa(X)=2$ and $X_0$ be one of its minimal models. Now we consider the case where the Albanese map $a: X\to A:=\Alb(X)$ factors through the Iitaka fibration $f_0:X_0\to S_0$. Since $S_0$ has rational singularities, the induced map $b_0:S_0\dashrightarrow \Alb(X)$ is indeed a morphism by \cite[Lemma 8.1]{Kawamata1985}. 

Let $\delta: S\to S_0$ be a resolution. After replacing $X$ by a higher birational model, we may assume that the induced maps $\psi:X\dashrightarrow X_0$ and $f:X\dashrightarrow S$ are morphisms. Let $b$ be the composition map $S\to A$. We obtain the following commutative diagram:
\begin{equation}
    \begin{tikzcd}\label{diagram}
        X_0 \arrow["f_0"]{r} & S_0 \arrow["b_0"]{r} & A\\
        X \arrow["f"]{r} \arrow["\psi"]{u}& S, \arrow["b"']{ur} \arrow{u}
    \end{tikzcd}
\end{equation}

Since $f$ is a fibration, $f^\ast: \Pic^0(S)\to \Pic^0(X)$ is injective. Since $a^\ast= f^\ast \circ b^\ast $ is an isomorphism, $f^\ast$ is surjective. Hence, $f^\ast$ is an isomorphism and, by duality, induces an isomorphism
\begin{equation}\label{eq:albx-albs}
    \Alb(X)\cong \Alb(S).
\end{equation}
In particular, $q(X)=q(S)$.

\subsection{A lower bound for the Iitaka volume}
In this subsection, we prove the following theorem:
\begin{theorem}\label{thm:non-iso-vol}
    Let $X$ be a smooth projective irregular threefold with $\kappa(X)=2$ such that the Albanese map factors through the Iitaka fibration. Then $\Iv(X)\geq \frac{1}{24}$.
\end{theorem}

 We prove the theorem by treating the cases $\operatorname{dim}a(X)=1$ and $\operatorname{dim}a(X)=2$ separately. Throughout the proof, let $H$ be an ample $\Q$-divisor on $S_0$ such that $K_{X_0}=f_0^\ast H$. In particular, for every $m\geq 0$, we have $f_{0\ast}\OO_{X_0}(mK_{X_0})\cong \OO_{S_0}(\lfloor mH\rfloor)$.
 
\begin{theorem}\label{thm: non-iso-alb-fiber-one}
    Under the assumption of Theorem \ref{thm:non-iso-vol}, suppose that $\operatorname{dim} a(X)=2$. Then $\Iv(X)\geq \frac{1}{4}$. 
\end{theorem}
\begin{proof}
    In the diagram (\ref{diagram}), each connected component of the general fiber of $X\to a(X)$ is an elliptic curve since $S\to a(X)$ is generically finite. Thus, $a_\ast \omega_X^{\otimes 2}\neq 0$. By Theorem \ref{thm:LPS-IT} and the isomorphism \eqref{eq:albx-albs}, $a_\ast \omega_X^{\otimes 2}$ satisfies IT$_0$. For all $P\in \Pic^0(A)$, \[h^0(a_\ast \omega_X^{\otimes 2}\otimes P)=\chi(a_\ast \omega_X^{\otimes 2}\otimes P)=\chi(a_\ast \omega_X^{\otimes 2})\] is independent of $P$. Denote this common value by $h$. By Lemma \ref{lem:IT-nonvanishing}, $h\geq 1$.

    By definition, the eventual map $\varphi_0$ of $\omega_X^{\otimes 2}$ (see Definition \ref{def-eventual}) satisfies \[\operatorname{dim}\operatorname{Im}\varphi_0\geq \operatorname{dim} a(X)=2.\] By \cite[Corollary 2.11]{Jiang2021}, $\operatorname{dim}\operatorname{Im}\varphi_0\leq \kappa(\omega_X^{\otimes 2})=2$. Hence, $\operatorname{dim}\operatorname{Im}\varphi_0=2$. For a positive integer $M$, consider the Cartesian diagram:
    \begin{center}
        \begin{tikzcd}
            X \arrow["a"]{r} &A\\
            X_M \arrow["a_M"]{r} \arrow["\mu_M"]{u} & A, \arrow["\pi_M"']{u}
        \end{tikzcd}
    \end{center}
    where $\pi_M$ is the multiplication-by-$M$ isogeny. By \cite[Lemma 2.10]{Jiang2021}, for general $Q\in \Pic^0(A)$ and sufficiently large $M>0$, the map $\varphi$ associated to the line bundle $\mu_M^\ast \omega_X^{\otimes 2}\otimes a_M^\ast Q$ is birationally equivalent to the base change of the eventual map $\varphi_0$ along $\mu_M$. In particular, the image of $\varphi$ is a surface.

    Base-changing diagram (\ref{diagram}) along $\pi_M$, we obtain:
    \begin{equation}
        \begin{tikzcd}
            X_{0,M} \arrow["f_{0,M}"]{r} & S_{0,M} \arrow["b_{0,M}"]{r} &A\\
            X_M \arrow["\psi_M"]{u} \arrow["f_M"]{r}& S_M, \arrow{u} \arrow["b_M"']{ur}
        \end{tikzcd}
    \end{equation}
    where $X_{0,M}$ is also minimal and admits the Iitaka fibration $f_{0,M}$. Let $K_{X_{0,M}}=f_{0,M}^\ast H_M$, where $H_M$ is an ample $\Q$-divisor on $S_{0,M}.$ We have \[2\mu_M^\ast K_X+ a_M^\ast Q=2K_{X_M}+a_M^\ast Q=\psi_M^\ast f_{0,M}^\ast (2H_M+b_{0,M}^\ast Q)+E_M,\] where $E_M\geq 0$ is $\psi_M$-exceptional.
    Thus, $\varphi$ descends to $\varphi_S$ on $S_{0,M}$ by Lemma \ref{lem: debarre}. By Lemma \ref{lem: lower bound for volume}, \[(2H_M+b_{0,M}^\ast Q)^2\geq h^0(S_{0,M},2H_M+b_{0,M}^\ast Q)-2.\]
     Notice that \[h^0(\omega_{X_{M}}^{\otimes 2}\otimes a_M^\ast Q)=\chi(a_{M\ast}\omega_{X_{M}}^{\otimes 2})=\chi(\pi_M^\ast a_\ast \omega_X^{\otimes 2})=d_M h,\]where $d_M$ is the degree of $\pi_M$. One can easily check that $H_M$ is the pullback of $H$ from $S_0$ to $S_{0,M}$. Therefore, \[\Iv(X)=(H^2)=\frac{1}{d_M}(H_M^2)=\frac{1}{4d_M}(2H_M+b_{0,M}^\ast Q)^2\geq \frac{d_Mh-2}{4d_M}\geq \frac{h}{4}-\frac{1}{2d_M}.\] Letting $M\to +\infty$, we obtain $\Iv(X)\geq \frac{1}{4}$. 
\end{proof}
\begin{remark}
Alternatively, Theorem \ref{thm: 14K} shows that $\lvert 6K_X\rvert$ defines the Iitaka fibration in this case. Let \[Y:=\overline{\varphi_{\lvert 6K_X\rvert}(X)}\subset \Pp^{P_6-1},\] which is birational to the Iitaka base $S_0$. Since $S_0$ admits a generically finite morphism onto a positive-dimensional subvariety of an abelian variety, $Y$ cannot be rational. Thus, $P_6\geq 4$ and $\operatorname{deg}Y\geq P_6-1\geq 3$. This yields the weaker bound
 \[\Iv(X)\geq \frac{\operatorname{deg}Y}{36}\geq \frac{1}{12}.\]
\end{remark}

\begin{lemma}\label{lem:base-g>1}
    Under the assumptions of Theorem \ref{thm:non-iso-vol}, suppose that there is a fibration $c$ from $S_0$ to a curve $C$ of genus $\geq 2$. Then $\Iv(X)\geq \frac{2}{3}$.
\end{lemma}
\begin{proof}
    Let $\sigma:=c\circ f_0$. By \cite[Corollary 1.3]{Patakfalvi2014}, $K_{X_0}-\sigma^\ast K_C$ is nef. Let $F$ (resp., $G$) be the general fiber of $\sigma$ (resp., $c$). Since $K_{X_0}-\sigma^\ast K_C=f_0^\ast(H-c^\ast K_C)$, we have that $H-c^\ast K_C$ is nef. Therefore,
    \begin{equation*}
        \begin{aligned}
          \Iv(X)&=(H^2)=(H-c^\ast K_C+c^\ast K_C)^2\\
          &\geq 2(H\cdot c^\ast K_C)= 2(2g(C)-2)\cdot \operatorname{deg}H\vert_G\\
          &\geq 4\Iv(F)\geq \frac{2}{3}.
        \end{aligned}
    \end{equation*}
Here the first inequality follows from the nefness of $H-c^\ast K_C$ and the equality $(c^\ast K_C)^2=0$ while the last follows from Theorem \ref{thm:Ivol 1/6}.
\end{proof}

\begin{theorem}\label{thm: non-iso-alb-fiber-two}
    Under the assumptions of Theorem \ref{thm:non-iso-vol}, suppose that $\operatorname{dim} a(X)=1$. Then $\Iv(X)\geq \frac{1}{24}$. 
\end{theorem}
\begin{proof}
    By Lemma \ref{lem:base-g>1}, we can assume that $a(X)=A$ is an elliptic curve. If $b_0$ has disconnected fibers, applying Lemma \ref{lem:base-g>1} to its Stein factorization gives the desired bound. We may therefore assume that $b_0$ is a fibration.

    Let $F$ (resp., $F_0$, $G_0$) be a general fiber of $a$ (resp., $a_0$, $b_0$). Since \[K_{F_0}=K_{X_{0}}\vert_{F_0}=(f_0\vert_{F_0})^\ast H \vert_{G_0},\] $F$ is a smooth surface of Kodaira dimension one. Let $E_m:=a_\ast \omega_{X}^{\otimes m}$ for $m\geq 1$. Choose $m\in \{2,3,4\}$ as in Lemma \ref{lem:1/24}, so $E_m$ is nonzero of rank $r_m>0$. Consider the natural multiplication map $\nu: E_{m}^{\otimes r_m}\to E_{r_m m}$ and let $W$ be its image. Then $W$ is nonzero and torsion-free, hence a vector bundle on $A$. By Lemma \ref{lem:HuZhangSlope} and the ampleness of $E_m$, \[\mu_{\max}(W)\geq \mu_{\min}(W)\geq \mu_{\min}(E_{m}^{\otimes r_m})=r_m\mu_{\min}(E_m)\geq r_m\cdot \frac{1}{r_m}=1.\] For any line bundle $L$ of degree one, $\mu_{\max}(W\otimes L^{-1})\geq 0$. According to Lemma \ref{lem:slope-nonvanishing}, there exists a line bundle $L^\prime$ of degree one such that
    \begin{equation*}
        \begin{aligned}
          h^0(mr_mH-b_0^\ast L^{\prime})&=h^0(mr_mK_{X_0}-a_0^\ast L^{\prime })  \\
          &=h^0(\omega_X^{\otimes mr_m}\otimes a^\ast L^{\prime -1})\\
          &=h^0(a_\ast \omega_X^{\otimes mr_m}\otimes L^{\prime-1})\\
          &\geq h^0(W\otimes  L^{\prime -1})\geq 1.
        \end{aligned}
    \end{equation*}
    Thus, $mr_mH-G_0$ is numerically $\Q$-effective and by Lemma \ref{lem:1/24} \[\Iv(X)=(H^2)\geq \frac{1}{mr_m}\operatorname{deg}H\vert_{G_0}=\frac{\Iv(F)}{mr_m }\geq \frac{1}{24}.\]
    
\end{proof}

\textit{Proof of Theorem \ref{thm:non-iso-vol}}: The result follows from Theorem \ref{thm: non-iso-alb-fiber-one} and Theorem \ref{thm: non-iso-alb-fiber-two}.

\subsection{The effective Iitaka fibration}
\begin{definition}[{\cite[Definition 2.6]{ChenHacon2007}}]
    Let $a:X\to A$ be a morphism from a smooth projective variety to an abelian variety, and let $D$ be a Cartier divisor on $X$. Let \[V^0(D):=\{P\in \Pic^0(A)\vert h^0(X, \OO_X(D)\otimes a^\ast P)\neq 0\}.\]
    We say that $D$ is \emph{full} if $V^0(D)=\Pic^0(A)$.
\end{definition} 
The following theorem is an analogue of \cite[Theorem 2.8]{ChenHacon2007} for varieties of intermediate Kodaira dimension:
\begin{theorem}\label{thm:sep-inter-kodaira}
    Let $X$ be a smooth projective variety of dimension $d$ and Kodaira dimension $1\leq \kappa\leq d-1$. Assume that there is a factorization $a=b\circ f$:
    \[ X\xrightarrow{f} Z\xrightarrow{b} b(Z)=a(X)\subset A,\]
    where $Z$ is smooth, $A$ is an abelian variety and $f$ is birationally equivalent to the Iitaka fibration. Assume that $a$ is non-trivial. Let $F$ be the general fiber of $a$. If $mK_X, nK_X, tK_X$ are full for some positive integers $m,n,t$, then 
    \begin{enumerate}
        \item $\lvert (m+n)K_X+a^\ast P \rvert$ separates two general points lying on distinct general fibers of $a$ for general $P\in \Pic^0(A)$.
        \item $\lvert (m+n+t)K_X+a^\ast P \rvert$ separates two general points lying on distinct general fibers of $a$ for all $P\in \Pic^0(A)$.
        \item If $\lvert (m+n)K_F \rvert$ defines the Iitaka fibration, then $\lvert (m+n)K_X+a^\ast P \rvert$ separates two general points $x,y\in F$ with $f(x)\neq f(y)$, for general $P\in \Pic^0(A)$.
        \item If $\lvert (m+n)K_F \rvert$ defines the Iitaka fibration, then $\lvert (m+n+t)K_X+a^\ast P \rvert$ separates two general points $x,y\in F$ with $f(x)\neq f(y)$, for all $P\in \Pic^0(A)$.
    \end{enumerate}
\end{theorem}
\begin{proof}
Let $X\to W\to A$ be the Stein factorization of $X\to A$. After a birational modification of both $X$ and $W$, we can assume that $W$ is also smooth. Notice that although \cite[2.7 General setting]{ChenHacon2007} is formulated for varieties of general type, its construction also works when $K_X$ is $W$-big. In our setting, $K_X$ is $W$-big since $X\to W$ factors through the Iitaka fibration. Using the same notation as \cite[2.7 General setting]{ChenHacon2007}, for any $m\geq 2$, $K_X+L_{m}$ and $mK_X$ are full if $P_m(G)\neq 0$, where $G$ is the general fiber of $X\to W$. In addition, if $mK_X$ and $nK_X$ are full, then for a general point $x\in X$, $x$ is not a base point of $\lvert K_X+L_{m+n}+a^\ast P\rvert $ for all $P\in \Pic^0(A)$.

    For two general points $x_1$, $x_2$ on $X$, we have seen that $x_i\notin \operatorname{Bs}\lvert K_X+L_{m+n}+a^\ast P \rvert$ for any $P\in \Pic^0(A)$. Thus, $x_i\notin \operatorname{Bs}\lvert (K_X+L_{m+n})\vert_F\rvert$. (1) follows from \cite[Corollary 2.4]{ChenHacon2007}.

    Now consider (3). Since $\lvert (m+n)K_F \rvert$ defines the Iitaka fibration, $\lvert (K_X+L_{m+n})\vert_F \rvert $ separates $x_1$, $x_2$. Then (3) follows from \cite[Corollary 2.4]{ChenHacon2007}.

    (2) and (4) follow by considering the morphism
    \[\lvert (m+n)K_X+a^\ast P_1 \rvert+\lvert tK_X+a^\ast P_2 \rvert\to \lvert (m+n+t)K_X+a^\ast (P_1+P_2)\rvert.\]
    Since $\lvert (m+n)K_X+a^\ast P_1\rvert $ separates $x_1$, $x_2$ for general $P_1\in \Pic^0(A)$ and $\lvert tK_X+a^\ast P_2 \rvert$ admits a section that is nonzero at both $x_1$ and $x_2$ for general $P_2\in \Pic^0(A)$, $\lvert (m+n+t)K_X+a^\ast P\rvert $ separates $x_1$, $x_2$ for all $P\in \Pic^0(A)$.

\end{proof}
We now establish an effective bound for the Iitaka fibration in this case:
\begin{theorem}\label{thm: 14K}
    Let $X$ be a smooth projective irregular threefold with $\kappa(X)=2$ such that the Albanese map $a: X\to A$ factors through the Iitaka fibration. Then $\lvert mK_X+a^\ast P\vert$ defines the Iitaka fibration for $m\geq 14$ and all $P\in \Pic^0(A)$.
\end{theorem}
\begin{proof}
After a birational modification, we can assume that \[X\xrightarrow{f} S\xrightarrow{b}b(S)=a(X)\subset A,\]
where $f$ is birationally equivalent to the Iitaka fibration.

Suppose first that $\operatorname{dim}a(X)=2$. Applying Theorem \ref{thm:sep-inter-kodaira} (2) (4) with $m=n=t=2$, we obtain that there is an open subset $U$ of $X$ such that $\lvert 6K_X+a^\ast P\rvert $ is defined at every point of $U$ and separates $x, y\in U$ with $f(x)\neq f(y)$. Since $(6K_X+a^\ast P)\vert_E \sim 0$ for a general fiber $E$ of $f$, the rational map defined by $\lvert 6K_X+a^\ast P\rvert $ is constant on each fiber $E$ meeting $U$. It therefore factors through $f$ as \[\psi_{\lvert 6K_X+a^\ast P\rvert }=\psi_S \circ f,\] where $\psi_S$ is a rational map on $S$. Since $\psi_S$ is injective on $f(U)$ which contains a dense open subset, $\psi_S$ is birational so $\lvert 6K_X+a^\ast P\rvert $ defines the Iitaka fibration. More generally, applying Theorem \ref{thm:sep-inter-kodaira} (2) (4) with $m=n=2$ and arbitrary $t\geq 2$, the same argument shows that $\lvert rK_X+a^\ast P\rvert $ defines the Iitaka fibration for every $r\geq 6$ and every $P\in \Pic^0(A)$.

    Now suppose that $\operatorname{dim}a(X)=1$ and let $F$ be the general fiber of the fibration in the Stein factorization of $a$. Using Lemma \ref{lem: surface-4-6-10} and Theorem \ref{thm:sep-inter-kodaira} (2) (4), we obtain the desired conclusion by choosing $(m,n,t)$ as follows:
    \begin{enumerate}
        \item For $r=14$, take $(m,n,t)=(6,4,4)$.
        \item For $r=15$, take $(m,n,t)=(6,4,5)$ if $P_5(F)\neq 0$, and $(m,n,t)=(3,8,4)$ otherwise.
        \item For $r=16$, take $(m,n,t)=(4,6,6)$.
        \item For $r=17$, take $(m,n,t)=(6,4,7)$ if $P_7(F)\neq 0$, and $(m,n,t)=(5,6,6)$ otherwise.
        \item For $r\geq 18$, take $(m,n,t)=(4,6,r-10)$.
    \end{enumerate}
    In each case, the same argument as above shows that $\lvert rK_X+a^\ast P\rvert $ defines the Iitaka fibration for every $P\in \Pic^0(A)$.
\end{proof}

\section{On the minimal Iitaka volume of abelian quotients}
In this section, we prove Theorem \ref{thm: geneisotrivialvollowerbd}, which gives the sharp lower bound of the Iitaka volume of generically isotrivial threefolds with abelian quotient group. 
This implies Theorem \ref{mainthm: iitaka-vol} (1) in the non-Albanese-factoring case by Corollary \ref{cor: chargeneisotrivial3fold}.

\begin{theorem}\label{thm: geneisotrivialvollowerbd}
Let $X$ be a smooth projective threefold with $\kappa(X)=2$. 
Suppose that 
$X\sim_{\text{bir}} (T\times E)/G$ is generically isotrivial, where $T,E,G$ are described in subsection \ref{subsec: geneiso}.
Assume that $G$ is abelian. Then 
\begin{itemize}
    \item[(1)] $\Iv(X)\geq \frac{1}{100}$;
    \item[(2)] $\Iv(X)=\frac{1}{100}$, if and only if 
    \[
 T\cong \{w^2=x^{10}+y^{10}+z^{5}\}\subset\Pp(1,1,2,5),\qquad G\cong \Z_{10} \times \Z_{10},
\]
and for any $g=(s,t)\in \Z_{10} \times \Z_{10},[x:y:z:w]\in T, e \in E$, one has
\[
 g.([x:y:z:w],e)=
([\zeta_{10}^sx:\zeta_{10}^ty:z:w],g_E.e),
\]
where $g_E$ is the induced action of $g\in G$ on $E$ by translations.
\end{itemize}

\end{theorem}
\begin{proof}
Denote by $S$ the minimal resolution of $T$. 
By \cite[Section 2.5 p.183]{Sym4fold08}, $\Aut(S)=\Aut(T)$ so that $G$
is also an automorphism group on $S$.

Since $G$ is abelian and $|G|\leq 100\vol(S)$ as in Theorem \ref{thm: abeliangroupbd},
 $\Iv(X)=\frac{\vol(S)}{|G|}\geq \frac{1}{100}$ by Corollary \ref{cor: chargeneisotrivial3fold} (1).
By Theorem \ref{thm: abeliangroupbd},
equality implies that $(V,p_g,q)=(1,2,0)$
and $|G|=100$.
Since $G$ is abelian and acts faithfully on $E$, $G$ acts on $E$ by translations by Lemma \ref{lem: bdQwithlinearpart}.
Hence we obtain (2) by applying Lemma \ref{lem: whenT12G100} on canonical model $T$ of $S$.
\end{proof}

In the following, we bound the order $|G|$ of finite abelian automorphism group on smooth minimal surface $S$ of general type.

Recall that the geometric genus, volume, irregularity and algebraic Euler characteristic of $S$ are denoted by $p_g,V,q,\chi$  respectively.
\subsection{Abelian automorphism group for minimal surfaces of general type}\label{subsection: abautogrp}
\begin{theorem}\label{thm: abeliangroupbd}
If $G$ is abelian, then $|G|\leq 100V$.
Equality holds only if $S$ is $(1,2)$-surface.
\end{theorem}
\begin{proof}[Proof of Theorem \ref{thm: abeliangroupbd}]
By Lemma \ref{lem: createGinvpencil}, there is an $n$-canonical $G$-invariant pencil generated by two semi-invariant vectors with canonical degree $\lambda\leq n$, where $n\in \{1,2\}$.
If this pencil is irrational, then $|G|\leq 54V$ by Proposition \ref{prop: irratpencilcase}. 
If this pencil is rational, then $|G|\leq 100V$ by Proposition \ref{prop: ratpencilcase}.
Moreover, equality holds only if $(V,p_g,q)=(1,2,0)$.
\end{proof}

In the following, we bound $|G|$ by  rational $G$-invariant pluricanonical pencil.
\begin{proposition}\label{prop: ratpencilcase}
If the $G$-invariant pencil defined in Lemma \ref{lem: createGinvpencil} is rational, then the following estimates hold:
\begin{enumerate}
\item[(1)] If $p_g\geq 2$ and $V\geq3$, or  $p_g=V=2$, then $|G|\leq 48V+24$;
\item[(2)] If $p_g= 3$ and $V=2$, then $|G|\leq 70V$; 
\item[(3)] If $p_g\geq2$ and $V=1$, then $|G|\leq100$, $|G|=100$ only if $p_g=2,q=0$;
\item[(4)] If $p_g\leq 1$ and $ V\geq 2$, then $|G|\leq 48V+24$;
\item[(5)] If $p_g=V=1$ and $q=0$, then $|G|\leq 88$.
\item[(6)] If $p_g=q=0$ and $V=1$, then $|G|\leq 70$.
\end{enumerate}
\end{proposition}

\begin{proof}
If $V=1$ and $p_g\geq2$, then we can take a canonical pencil by Lemma \ref{lem: createGinvpencil}
so that $\lambda\leq 1 $. 
The first term in Lemma \ref{lem: calcu} is at most $48V$. 
The second term in Lemma \ref{lem: calcu} is at most
\[
\frac{5}{2}\bigl(12\chi-V+5V\bigr).
\]
  Noether's inequality gives $p_g=2$ and $\chi\leq 3$ so that the second term is not greater than $\frac{5}{2}(36-1+5)=100$, and equality holds only if $p_g=2,\chi=3$ so that $q=0$.
 This proves (3).

In the following, we apply Lemma \ref{lem: resolintnumcal} and Lemma \ref{lem: calcu} and use notations there.
Write $s:=g-1$.
By Lemma \ref{lem: resolintnumcal} (1) and (4),
$1\leq s\leq \frac{1}{2}(\sum_i a_i+\lambda V)\leq \frac{\lambda(\lambda+1)}{2}\,V$.
Denote by $U_1,U_2$ the first and the second term in \eqref{eq: ineqthmLX} respectively.

Firstly we bound $U_2$.
By the third line in \eqref{eq: ineqGef} and the Noether formula $\chi_{top}=12\chi-V$ on $S$, we have 
\begin{equation}\label{eq: U2}
U_2\leq \left(2+\frac1{s+1}\right)(12\chi+(\lambda^2-1)V+4s).    
\end{equation}
This is a convex function of $s\in[1,3V]$, so we can take values at end points to obtain the maximal value.

If $p_g\geq 2$, then $\lambda\leq 1$ by Lemma \ref{lem: createGinvpencil}.
Then  
\begin{equation*}
 U_2\leq \left(2+\frac1{s+1}\right)(12\chi+4s) \leq  30\chi+10V.
\end{equation*}
by \eqref{eq: U2} and $1\leq s\leq V$.
By Noether's inequality, $V\geq 2p_g-4$ implies $\chi\leq \frac{V}{2}+3$.
Moreover, we have $\chi\leq 3$ when $p_g=2$.
So
\begin{equation}\label{eq: U2bdpg2}
 U_2\leq 30\chi+10V\leq \begin{cases}
 25V+90\leq 48V+24 & \text{if }V\geq 3;    \\
 10V+90\leq 48V+24 &  \text{if }V=p_g=2;    \\
 25V+90\leq 70V &\text{if }V=2,p_g=3.    \\
 \end{cases}
\end{equation}

If $p_g\leq 1$, then $\lambda\leq 2$ by Lemma \ref{lem: createGinvpencil}.
So
\begin{equation}\label{eq: U2bd}
 U_2\leq\max\left\{
 30\chi+\frac{15V}{2}+10,
 30V+24\chi+5+\frac{12\chi-5}{3V+1}
 \right\} \leq \begin{cases}
 48V+24& \text{ if }V\geq 2,\chi\leq 2; \\
 70 &\text{ if }V=\chi=1; \\
 88 &\text{ if }V=1,\chi=2.    \\
 \end{cases}
\end{equation}

Then we bound $U_1$.
If $1\leq s \leq V$, then we can discard $r$ in the second line of \eqref{eq: ineqGKf} in Lemma \ref{lem: calcu}, and get
\begin{equation}\label{eq: U1bd}
 U_1=\left(2+\frac{4}{s}\right)\,\left(V-r+8s\right)\leq \left(2+\frac{4}{s}\right)\,\left(V+8s\right)\leq \max\{6V+48,18V+36\}< 48V+24.
\end{equation}
If $p_g\leq 1$ and $V\leq s \leq 3V$.
By Cauchy inequality, we have
\begin{equation*}
\begin{aligned}
 4V\cdot r &\geq \lambda^2V\cdot r \geq \left(\sum_i a_i^2\right)\cdot r \\
 &= \left(\sum_i a_i^2\right)\cdot \left(\sum_i 1 \right) \\
 &\geq \left(\sum_i a_i\right)^2
 \geq (2s-\lambda V)^2 \\
 &\geq 4(s-V)^2,
\end{aligned}    
\end{equation*}
where the last inequality holds when $s\geq V$.
So $r\geq \frac{(s-V)^2}{V}$ and 
\begin{equation}\label{eq: U1bd2}
 U_1=\left(2+\frac{4}{s}\right)\,\left(V-r+8s\right)\leq20s-\frac{2s^2}{V}+40-\frac{4s}{V}\leq 42V+28.
\end{equation}
as $V\leq s \leq 3V$.
By inequalities \eqref{eq: U2bd}, \eqref{eq: U1bd}, and \eqref{eq: U1bd2}, we obtain 
\[
|G|\leq \max\{U_1,U_2 \}\leq \begin{cases}
     88 & \text{ if }V=1,\chi=2;                   \\
    70 & \text{ if }V=1,\chi=1; \\
    48V+24 & \text{ otherwise } \\
\end{cases}
\]
for $p_g\leq 1$.
so that $(4)\sim(6)$ have been proved.

By inequalities \eqref{eq: U2} and \eqref{eq: U1bd}, and \eqref{eq: U1bd2},
we obtain 
\[
|G|\leq \max\{U_1,U_2 \}\leq  \begin{cases}
     70V& \text{ if }V=2,p_g=3;                   \\
    48V+24 & \text{ otherwise } \\
\end{cases}  \\
\]
for all $S$ in $(1)\sim (2)$.
Hence we complete the proof.
\end{proof}

Then we bound $|G|$ by irrational $G$-invariant pluricanonical pencil.
\begin{proposition}\label{prop: irratpencilcase}
Let $S$ be a smooth minimal surface of general type with a smooth model $S'$.
Let $G\leq \Aut(S)$ be a finite abelian subgroup.
Suppose that there is a Stein factorization $h:S'\xrightarrow[]{f} B \xrightarrow[]{g} \Pp^1$ of a $G$-invariant $n$-canonical pencil, where $n\in\{1,2\}$.

If $g(B)>0$,
then $V\geq 2$ and $|G|\leq \min\{54V, 48V+12\}$.
\end{proposition}

\begin{proof}
Write $k:=\deg g \geq 2$. 
Since $g(B)>0$, we might contract $f$-vertical $(-1)$-curves and assume $S'=S$.
Then 
\begin{equation}\label{eq: kg-1}
    V\geq \frac{n}{2}K_S^2\geq \frac{k}{2}(K_S\cdot F)=k(g-1)\geq 2
\end{equation} by the adjunction formula and $n\leq 2,k\geq 2$.

Denote $Q:=\im(G \to \Aut(B))$ and $K:=\ker(G\to \Aut(B))$
with $|G|=|K|\cdot |Q|$.
The kernel $K$ acts faithfully on a general fiber, so Lemma \ref{lem: curveautos} implies $|K|\leq 4g+4$.

If $b=g(B)\geq2$, then 
$|G|\leq18V $ by Lemma \ref{lem: bdGwhendimalb=1andqgeq2}. 
Now suppose $b=1$. 
If $Q$ contains a non-translation, then $|Q|\leq 9$ by Lemma \ref{lem: bdQwithlinearpart} so that 
\[
|G|\leq9(4g+4)=36(g-1)+72\leq18V+72\leq \min\{54V, 48V+12\}
\]
as $V\geq 2$.
If $Q$ consists of translations, then Lemma \ref{lem: Qleqk} implies $|Q|\leq k$ and
\[
|G|\leq k(4g+4)=4k(g-1)+8k\leq 12k(g-1)\leq 12V
\]
by \eqref{eq: kg-1}.
\end{proof}

Moreover, in order to prove Theorem \ref{thm: geneisotrivialeffIitakafib} by Proposition \ref{prop: nVbdimply32K}, we need the following refined bound of $|G|$ 
when $p_g=q=1$.  
\begin{lemma}\label{lem: pgq1bdG}
If $p_g(S)=q(S)=1$, then $|G|\leq 36V$.
\end{lemma}
\begin{proof}
Since $q(S)=1$, we can write 
$f:S \to B$ as the Albanese morphism to an elliptic curve $B$.
Let $L:=f_*\omega_S=f_*\omega_{S/B}$.
Then $L$ is nef and locally free by semi-positivity theorem (See \cite{Kawamata82}).
It has degree $\deg L=\chi(\OO_S)=1$.

Let $A\subseteq L$ be the positive-slope part of the Harder–Narasimhan filtration of $L$
(See subsection \ref{subsec: HNfiltration}).
Then 
$\deg A=1, r:=\operatorname{rk}A\geq 1$,
and $A$ is stable of slope $\frac{1}{r}$.

Denote $H:=\ker(G \to \Aut(B))$ and $G/H\cong \operatorname{im}(G \to \Aut(B))\leq \Aut(B)$ is an abelian automorphism group of an elliptic curve.
Then the image of $G/H$ in $ \Aut(B)$ has no translation part $t$.
Indeed, if there is a translation $t$ by contrary, then $t$ has no fixed point on $B$. 
Choose a lift $g\in G$ of $t$.
Since $t$ has no fixed points on $B$, $g$ has no fixed points on $S$.

Since $p_g=q=1$, the traces of $g$ on $H^0(S,\OO_S)$ and $H^1(S,\OO_S)$ are both equal to one, but its trace on one-dimensional $H^2(S,\OO_S)$ is $\zeta\neq0$, where $\zeta $ is a root of unity.
This contradicts the holomorphic Lefschetz fixed point theorem (See \cite[Theorem 4.12]{AtiyahBott68}).
So $G/H$ has only linear part and $|G/H|\leq 6$.

If $r=1$, then the evaluation morphism
$f^*A\to \omega_S$ is nonzero injective morphism between line bundles.
So
$K_S-f^*A$ is effective. 
Denote by $g$ the genus of general fiber of $f$.
Then $V\geq K_S\cdot f^*A=2g-2$.

The finite group $H$ acts faithfully on general fiber $F$, so
\[
 |H|\leq4g+4\leq2V+8
\]
by Lemma \ref{lem: curveautos}.
Since $|G/H|\leq 6$, one has
\[
 |G|\leq 6|H| \leq 12V+48\leq36V,
\]
where the last inequality holds as $V\geq2p_g= 2$ by \cite[Theorem~6.1]{Debarre1982}.

If $r\geq 2$, consider the following fiber product along the isogeny $[n]:B_n \to B, b\mapsto nb$
\[
\begin{tikzcd}
   S_n=S\times_B B_n &  S \\
   B_n   & B    \\
       \arrow[from=1-1, to=1-2,"\pi_n"]
    \arrow[from=1-1, to=2-1,"f_n"]
    \arrow[from=1-2, to=2-2,"f"]
    \arrow[from=2-1, to=2-2,"{[n]}"], 
\end{tikzcd}
\]
where $B_{n}\cong B$ is elliptic curve.

Since $f$ has connected fibers, $S_n$ is connected.
The map $\pi_n$ is \'{e}tale with $\deg\pi_n=N:=n^2$, so $S_n$ is smooth and minimal.
Moreover, the pullback of a semistable bundle by an \'{e}tale isogeny is
semistable.
So $A_n:=[n]^*A$ is semistable of rank $r$,  $h^0(B_n,A_n)=\deg(A_n)=N$, and slope $N/r$.
Let $a_n:S_n \dashrightarrow \Sigma_n\subset \Pp(W_n^{*})$ be the rational map induced by sublinear system $W_n:=H^0(B_n,A_n) \subseteq H^0(S_n,\omega_{S_n})$.

By taking large $n$ with $\deg(A_n\otimes\OO_{B_n}(-b-b'))=N-2r=n^2-2r>0$, one has 
\[
H^1(B_n,A_n\otimes\OO_{B_n}(-b-b'))=0
\]
for any two distinct closed points $b,b'\in B_n$.
So 
$H^0(B_n,A_n) \to H^0(\{b\},A_n|_{b}) \oplus H^0(\{b'\},A_n|_{b'})$ is surjective,
where $h^0(\{b\},A_n|_{b}) =h^0(\{b'\},A_n|_{b'})=\operatorname{rank}(A_n)=r \geq 2$ so that $\dim(a_n|_{f_n^{-1}(b)})\geq 1$ and $a_n$ separates distinct fibers $f_n^{-1}(b),f_n^{-1}(b')$. 
Thus $a_n$ is generically finite. 

Note that $A$ is stable bundle on $B$ and $\operatorname{End}(A)=\C$.
Then there is a character $\chi:H\to \C^*$ such that $h|_A=\chi(h)\cdot\operatorname{id}_A$ and $ a_n\circ h=a_n$ for any $h\in H$.

Since $H$ acts faithfully on general fibers, $|H|\leq \deg a_n$.
So the image $\Sigma_n$ of $a_n$ is a surface in $\Pp^{N-1}$ and
\[
 NV=\vol(S_n)
       \geq M_n^2
       \geq |H|\deg\Sigma_n
       \geq |H|(N-2),
\]
where $M_n$ is movable part of some resolution of linear system $W_n$.
Then $|H|\leq V$ by taking $N\to +\infty$.
Moreover, we have $|G/H|\leq 6$.
Hence $|G|\leq 6V$ when $r\geq 2$.
\end{proof}
\section{On the bound for Iitaka fibrations of translation quotients}\label{section: effIitgeneisotrivial}
In this section, we prove sharp Iitaka fibration in Theorem \ref{thm: geneisotrivialeffIitakafib} for generically isotrivial threefolds with the quotient group acting on elliptic-curve factor by translations. 

This implies Theorem \ref{mainthm: eff-iitaka} (1) in the non-Albanese-factoring case by Corollary \ref{cor: chargeneisotrivial3fold}.

\begin{theorem}\label{thm: geneisotrivialeffIitakafib}
Let $X$ be a smooth projective threefold with $\kappa(X)=2$. 
Suppose that 
$X\sim_{\text{bir}} (T\times E)/G$ is generically isotrivial, where $T,E,G$ are described in subsection \ref{subsec: geneiso}.
Assume $G$ acts on $E$ by translations.
Then
\begin{itemize}
    \item[(1)] $|mK_X|$ defines the Iitaka fibration for every integer $m\geq 32$;
    \item[(2)] The following conditions are equivalent.
    \begin{itemize}
        \item[(a)]  $|31K_X|$ does not define the Iitaka fibration;
        \item[(b)]  $h^0(X,31K_X)=1$;
        \item[(c)]   $X\sim_{\text{bir}} (T\times E)/G$ is generically isotrivial as in subsection \ref{subsec: geneiso}, where
    \[
 T\cong \{w^2=x^{10}+y^{10}+z^{5}\}\subset\Pp(1,1,2,5),\qquad G\cong \Z_{10} \times \Z_{10},
\]
and for any $g=(s,t)\in \Z_{10} \times \Z_{10},[x:y:z:w]\in T, e \in E$, one has
\[
 g.([x:y:z:w],e)=
([\zeta_{10}^sx:\zeta_{10}^ty:z:w],e+t_g),
\]
where the induced action of $g \in G$ on $E$ defined as $g: e \mapsto e+t_g$ is a translation.
    \end{itemize} 
\end{itemize}
\end{theorem}
\begin{proof}[Proof of Theorem {\ref{thm: geneisotrivialeffIitakafib}}]
Let $S$ be the minimal resolution of $T$.
Denote the volume, geometric genus and irregularity of $S$ by $V,p_g,q$ respectively.

If $V \geq 2$ and $(V,p_g,q) \neq (2,3,0)$, then these surfaces satisfy one of the assumptions in Corollary \ref{cor: mostcases} so that $|mK_X|$ defines the Iitaka fibration for any $m\geq 31$.

If $(V,p_g,q)=(2,3,0)$, then the ratios of sections in $H^0(S,mK_S)^G$ generate $\C(S)^G$ for any $m\geq 31$ by Proposition \ref{prop: invariantsectionring-23}.
Then $|mK_X|$ defines the Iitaka fibration for any $m\geq 31$ by Lemma \ref{lem: reducemKbirattoS/G}.

If $(V,p_g,q)=(1,2,0)$ and $|G|<100$, then $|mK_X|$ defines the Iitaka fibration for any $m\geq 31$ by Proposition \ref{prop: invariantsectionring-12}.

If $(V,p_g,q)=(1,1,0)$, then $|G|\leq 36$
by Proposition \ref{prop: pg1-volume1-abelian} and $|mK_X|$ defines the Iitaka fibration for any $m\geq 31$ by Lemma \ref{lem: Godsurf}.

If $(V,p_g,q)=(1,0,0)$, then $|mK_X|$ defines the Iitaka fibration for any $m\geq 31$ by Proposition \ref{prop: Godsurf}.

If $(V,p_g,q)=(1,2,0)$ and $|G|=100$, then  $T,G$ satisfy the conclusions in Lemma \ref{lem: whenT12G100} (2).
By calculations of $G$-invariant canonical rings in Proposition \ref{prop: 31KdoesnotdefIitfib} of such $T$, we obtain $h^0(X,31K_X)=1$, but $|mK_X|$ defines the Iitaka fibration for any $m\geq 32$.

In summary, $|mK_X|$ defines the Iitaka fibration for any $m\geq 32$  so that we have proved (1).
As for (2), $|31K_X|$ does not define the Iitaka fibration if and only if $X$ lies in the case with $(V,p_g,q)=(1,2,0)$ and $|G|=100$. 
Then conditions (a),(b) and (c) of (2) are equivalent by the arguments of last paragraph.
\end{proof}

We fix the following notation.
\begin{definition}\label{defn: XS/GYP}
Let $S$ be a smooth minimal surface of general type.
Denote the volume, geometric genus, irregularity, algebraic Euler characteristic of $S$ by $V,p_g,q,\chi$ respectively.
Let $G\leq \Aut(S)$ be a finite abelian subgroup.
Then $\pi:S\to S/G$ is a finite morphism to normal projective variety with klt rational singularities.
Denote by $B$ an effective $\Q$-divisor on $S/G$ such that $K_S=\pi^*(K_{S/G}+B)$.
Denote by $\mu:Y\to S/G$ a resolution of $S/G$.
Set $P:=\mu^*(K_{S/G}+B)$ to be a $\Q$-Cartier nef and big divisor on $Y$.
\end{definition}
\begin{lemma}\label{lem: reducemKbirattoS/G}
Let $X$ be a smooth projective threefold with $\kappa(X)=2$. 
Suppose that 
$X\sim_{\text{bir}} (T\times E)/G$ is generically isotrivial, where $T,E,G$ are described in subsection \ref{subsec: geneiso}.
Assume that $G$ acts on $E$ by translations.

Denote by $S$ the minimal resolution of $T$.
Then
\begin{itemize}
    \item[(1)] $T/G$ is the Iitaka base of $X$;
    \item[(2)] Fix a positive integer $m\geq 2$.
Then 
\begin{itemize}
    \item[(a)] the ratios of sections in $H^0(S,mK_S)^G$ generate $\C(S)^G$, if and only if $|mK_X|$ defines the Iitaka fibration; 
    \item[(b)] if $|K_Y+\lceil (m-1)P \rceil|$ induces a birational map, then $|mK_X|$ defines the Iitaka fibration.
\end{itemize}
\end{itemize}
\end{lemma}
\begin{proof}
Since
$G$ acts on $E$ by translations,  $H^0(X,tK_X)\cong (H^0(S,tK_S)\otimes H^0(E,tK_E))^G=H^0(S,tK_S)^G$ holds for every $t\geq 0$.
By \cite[Theorem 1.2]{BirkarZhang16}, $|kK_X|$ defines the Iitaka fibration for sufficiently divisible $k\gg 0$.
So the Iitaka model of $X$ is 
\[
\operatorname{Proj} \bigoplus_{l\geq0}^{\infty} H^0(X,lkK_X)=\operatorname{Proj}\bigoplus_{l\geq0}^{\infty} H^0(S,lkK_S)^G=\left( \operatorname{Proj}\bigoplus_{l\geq0}^{\infty} H^0(S,lkK_S) \right)/G =T/G
\]
for sufficiently divisible $k$, as the canonical model $T$ of $S$ is normal.
Moreover, the function field of $T/G$ equals $\C(T)^G$.

Assertion (a) of (2) follows, because the function field of the image of the pluricanonical map is generated by ratios of its defining sections.

Now we prove assertion (b) of (2).
Let $Z$ be the normalization of the main component of $S\times_{S/G} Y$.
Denote by $\pi_Y: Z\to Y$.
By Hurwitz formula, there is an effective $\Q$-divisor $B_Y$ on $Y$, whose coefficients belong to
 $\{ 1-\frac{1}{e} |e\in \N_{\geq 2}\}$ such that
 \[
 \pi_Y^*(K_Y+B_Y)=K_Z.
 \]
 Then $\lfloor mB_Y\rfloor =\lceil (m-1)B_Y\rceil $ holds by $\lfloor m(1-\frac{1}{e}) \rfloor=\lceil (m-1)(1-\frac{1}{e})\rceil$ for any $e\in\N_{\geq 1}$.
 
Since $K_{Z/S}\geq 0$, there is an effective $\mu$-exceptional $\Q$-divisor $N_0$ such that
\begin{equation*}
 K_Y+B_Y=P+N_0,
\end{equation*}
where $P$ is defined in Definition \ref{defn: XS/GYP}.
Then for every integer $m\geq 2$, we have
\begin{align*}
H^0(X,mK_X)&=H^0(S,mK_S)^G=H^0(S/G,mK_{S/G}+\lfloor{mB}\rfloor)=H^0(Y,mK_Y+\lfloor{mB_Y}\rfloor)\\\ 
 &=H^0(Y,mK_Y+\lceil (m-1)B_Y\rceil)=H^0(Y,K_Y+\lceil{(m-1)(P+N_0)}\rceil)    \\
 &
 \supseteq H^0(Y, K_Y+\lceil{(m-1)P}\rceil).
\end{align*}
As $Y$ is a birational model of $T/G$, $|mK_X|$ defines the Iitaka fibration when $| K_Y+\lceil{(m-1)P}\rceil|$ induces a birational map.
\end{proof}
\subsection{Upper bound of group order implies effective Iitaka fibration}\label{subsubsec: verygencurves}

In this part, we prove Proposition \ref{prop: nVbdimply32K}, which implies  sharp Iitaka fibrations for generically isotrivial $X\sim_{\text{bir}}S\times E/G$ with bounded quotient group $G$.

To prove Proposition \ref{prop: nVbdimply32K}, we estimate the degree of the quotient map $\pi:S\to S/G$ along curves passing through very general points.
This leads to the following notations.
\begin{definition}\label{defn: verygencurv}
Denote
\[
 \Sigma_S:=\bigcup_{1\ne\gamma\in G}\operatorname{Fix}_S(\gamma),
\]
where $\operatorname{Fix}_S(\gamma):=\{s\in S|\gamma.s=s \}$ is a closed subset of fixed points.
A subset $\mathcal S\subsetneq Y$ is called an \emph{admissible exceptional
set}, if it is a countable union of proper closed subsets containing:
\begin{enumerate}
\item[(1)] $\operatorname{Supp} P\cup\operatorname{Exc}(\mu)\cup\mu^{-1}(\pi(\Sigma_S))$;
\item[(2)] every integral curve $C$ on $Y$ with $C^2<0$;
\item[(3)] $\mu^{-1}(\pi(D))$ for every integral curve $D\subset S$ with
$D^2<0$, or the normalization of $D$ has genus at most one.
\end{enumerate}
For an integral curve $C\subset Y$ with $C\cap(Y\setminus \mathcal S)\neq\varnothing$,
let $C_S=\mu(C)$ and let $\Gamma$ be an irreducible component of
$(\pi^{-1}(C_S))_{\mathrm{red}}$. Write
\[
\nu:\widetilde\Gamma\longrightarrow\Gamma,\qquad
 H:=\{\gamma\in G:\gamma(\Gamma)=\Gamma\},\qquad h:=|H|,
 \qquad r:=K_S\cdot\Gamma,\qquad a:=\Gamma^2,
\]
where $\nu$ is the normalization.
\end{definition}
\begin{lemma}\label{lem: excsetexists}
The admissible exceptional set $\mathcal S$ exists on $Y$.
\end{lemma}
\begin{proof}
We show that curves of $(1)\sim (3)$ inside the admissible exceptional set $\mathcal S$ are all countable union of integral curves. 

By definition, (1) is a finite union of integral curves on $Y$. 

An integral curve $C$ on a smooth projective surface $Y$ with $(C^2)<0$ is rigid as
$H^0(C,\OO_C(C))=0$.
So such curves in (2) form a countable set.

Since $S$ is of general type, the union of integral curves whose normalization have genus at most one, forms a countable union of proper closed subsets of $S$.
Taking its image under a finite quotient $\pi$ and inverse image of birational morphism $\mu$ preserves this property.
So curves in (3) still form a countable union of curves.
\end{proof}
\begin{lemma}\label{lem: curvebound}
For every $C,\Gamma$ in Definition \ref{defn: verygencurv}, we have
\begin{gather}
 C^2\geq 0,\quad r\in\Z_{>0},\quad 0\leq a\leq r^2/V,
 \quad r+a\in2\Z,\quad h\mid |G|,\label{eq:curvebasic}  \\ 
 P\cdot C=r/h,\qquad
 h\leq\min\{|G|,\,2r+2a+8\}
 \leq\min\left\{|G|,2r+\frac{2r^2}{V}+8\right\}.\label{eq:curvebound}
\end{gather}
If $V\geq 2$, then $r\geq 2$.
\end{lemma}
\begin{proof}
Since $C \cap(Y\setminus \mathcal S)\neq\varnothing$, one has $C^2\geq 0$, $a=\Gamma^2\geq 0$, and
$g(\widetilde\Gamma)\geq 2$ by definition.

Since $K_S$ is nef and big and $a=\Gamma^2\geq 0$, $r=(K_S\cdot \Gamma)>0$.

The Hodge index theorem implies $a\leq  \frac{(K_S\cdot \Gamma)^2}{K_S^2}=\frac{r^2}{V}$.

By adjunction formula,
$r+a=K_S\cdot \Gamma+\Gamma^2=2p_a(\Gamma)-2\in2\Z$.

The group $H$ acts on $\Gamma$ faithfully
by definition (1) of $\mathcal S$.
Moreover, $\pi$ is \'etale over the generic point of $C$ so that
$\deg(\Gamma\to C)=h$.
Then projection formula implies $r=h(P\cdot C)$ and $(P\cdot C)=\frac{r}{h}$.

Since $(Y\setminus\mathcal{S})\cap \mu^{-1}(\pi(\Sigma_S))=\varnothing$, $H$ acts on $\widetilde{\Gamma}$ faithfully.
By $H$ is abelian and $g(\widetilde\Gamma)\geq 2$, Lemma
\ref{lem: curveautos} implies
\[
 h\le4g(\widetilde\Gamma)+4
 \le4p_a(\Gamma)+4=2r+2a+8 \leq 2r+\frac{2r^2}{V}+8.
\]

We have $h\mid |G|$ as $H\leq G$. 

Finally, if $V\geq2$ and $0<r=1$,
then $a\leq \frac{r^2}{V}= 1/V<1$, so $a=\Gamma^2=0$ and $r+a=1$.
This contradicts the result $r+a\in2\Z$.
\end{proof}

\begin{proposition}\label{prop: nVbdimply32K}
If $V\geq 2$ and $|G|\leq 48V+24$,
then $|mK_X|$ defines the Iitaka fibration for any $m\geq 31$.
\end{proposition}
\begin{proof}
By Lemma \ref{lem: excsetexists}, there are curves $C\in Y,\Gamma\in S$ defined in Definition \ref{defn: verygencurv} and characterized in Lemma~\ref{lem: curvebound}.

Firstly we claim that $h\leq 12r$.
Indeed, if $2\le r\le4V+1$, the convex polynomial
$r^2-5Vr+4V \leq 0$, as its values at its endpoints are respectively
    $4-6V$ and $-4V^2+7V+1$, both are negative for $V\geq 2$.
So \eqref{eq:curvebound} implies $h\leq 2r+\frac{2r^2}{V}+8\leq  12r$.
If $r\geq 4V+2$, then $h\leq |G|\leq 48V+24\le12r$.

Then we apply Lemma \ref{lem: criterionforsepa}.
Fix $t:=m-1\geq30$ and set $L:=tP$. By $h\leq 12r$ and $L^2=\frac{t^2V}{|G|}\geq \frac{900V}{48V+24} \geq \frac{900}{60}$, we have
\[
 \beta=L^2\geq\frac{900}{60}=15>8,
 \qquad c=L\cdot C\geq\frac{30}{12}> 
        \frac{4}{1+\sqrt{\frac{7}{15}}}\geq\frac{4}{1+\sqrt{1-\frac{8}{\beta}}}.
\]
By Lemma~\ref{lem: criterionforsepa}, $|K_Y+\lceil tP \rceil|$ separates two distinct very general points for $t\geq 30$. 
This implies $|mK_X|$ defines the Iitaka fibration for any $m\geq 31$ by Lemma \ref{lem: reducemKbirattoS/G} (2). 
\end{proof}
\begin{corollary}\label{cor: mostcases}
Let $X\sim_{\text{bir}}(T\times E)/G$ be a smooth generically isotrivial threefold.
Assume $G$ acts on $E$ by translations.
Denote by $S$ the minimal resolution of $T$ with invariants $(V,p_g,q):=(K_S^2,p_g(S),q(S))$.

Then $|mK_X|$ defines the Iitaka fibration for any $m\geq 31$, if one of the following holds.
\begin{itemize}
    \item[(1)] $q>0$;
    \item[(2)] $q=0$, $p_g\leq 1$ and $V\geq 2$;
    \item[(3)] $q=0$, $p_g\geq 2$ and $V\geq 3$, or $p_g=V=2$.
\end{itemize}
\end{corollary}
\begin{proof}
Assume condition (1) holds.
If $p_g=q=1$, then Lemma \ref{lem: pgq1bdG} implies $|G|\leq 36V<48V+24$.
If $p_g\geq 2$, then $V\geq 2p_g\geq 4$ by \cite[Theorem~6.1]{Debarre1982}.
There is a $G$-invariant canonical pencil by Lemma \ref{lem: createGinvpencil}.
If this pencil is irrational, then $|G|\leq 48V+24$ by Proposition \ref{prop: irratpencilcase}.
If this pencil is rational, then $|G|\leq 48V+24$ by Proposition \ref{prop: ratpencilcase} (1).

Assume condition (2) or (3) holds, then there is a $G$-invariant canonical or bicanonical pencil by Lemma \ref{lem: createGinvpencil}.
Since $q=0$, this pencil has to be rational.
So $|G|\leq 48V+24$ by Proposition \ref{prop: ratpencilcase} $(1),(4)$.

Hence $|G|\leq 48V+24$ holds for all above cases and $|mK_X|$ defines the Iitaka fibration for any $m\geq 31$ by Proposition \ref{prop: nVbdimply32K}.
\end{proof}

Note that Corollary \ref{cor: mostcases} covers all the possible $S$ with $V\geq 2$ and $(V,p_g,q)\neq (2,3,0)$.
In the following parts, we deal with the remaining cases with $V= 1$ or $(V,p_g,q)=(2,3,0)$ respectively.
.
\subsection{On numerical Godeaux surfaces}
In this part, we prove Proposition \ref{prop: Godsurf}, which gives effective Iitaka fibration for generically isotrivial case with a $(1,0)$-surface factor.
\begin{lemma}\label{lem: Godsurf}
Let $X\sim_{\text{bir}}(T\times E)/G$ be a smooth generically isotrivial threefold.
Assume $G$ acts on $E$ by translations and $|G|\leq 70$.
Denote by $S$ the minimal resolution of $T$.

If $\vol(T)=1$,
then one of the following holds.
\begin{itemize}
    \item[(1)] $|mK_X|$ defines the Iitaka fibration for any $m\geq 31$;
    \item[(2)] there is a $G$-invariant curve $\Gamma\subset S$ passing through very general points, such that
    \begin{equation}\label{eq:fourorders}
 K_S\cdot\Gamma=5,\qquad |G|\in\{63,64,65,66,67,68\}.
\end{equation}
\end{itemize}
\end{lemma}
\begin{proof}
Fix $t:=m-1\geq 30$ and put $L=tP$. Then
\[
 \beta=L^2=\frac{t^2}{|G|}\geq \frac{900}{70}>8.
\]
For a curve passing through very general points in Lemma~\ref{lem: curvebound},
\begin{equation}\label{eq:godeauxcurvebound}
 h\leq\min\{70,2r^2+2r+8\},\qquad h\mid |G|.
\end{equation}
If $r\neq 5$, then $h\leq 12r$.
Indeed, $h\leq 12r$ follows by direct calculation
for $r=1,2,3,4$, and by $h\leq 70$ for $r\geq 6$. 
So we might assume $r=(K_S\cdot \Gamma)=5$.

Then we apply Lemma \ref{lem: criterionforsepa}.
Denote $t:=m-1\geq 30$ and $L:=tP$.
By calculations, we can reduce the assumptions of this lemma to
\begin{equation}\label{eq: cbeta}
 \beta=L^2=\frac{t^2}{|G|}\geq\frac{900}{70}>8,
 \qquad  c=L\cdot C>
        \frac{4}{1+\sqrt{1-\frac{8}{\beta}}} \quad \text { or } \quad c-\frac{c^2}{\beta}>2,    
\end{equation}
where $(L\cdot C)=\frac{t}{h}(K_S\cdot \Gamma)\geq \frac{150}{h}$.

If $L\cdot C \geq 4$, then \eqref{eq: cbeta} holds clearly.
We might assume $c=L\cdot C<4$.
By $r=5$, $h\leq 2r^2+2r+8= 68$ and $c\geq 150/68>2$. 

If $h<|G|$, $h\mid |G|$
implies $h\leq \frac{|G|}{2}\leq 35$, hence $c\geq 150/35>4$ and it satisfies \eqref{eq: cbeta}. 
By Lemma \ref{lem: criterionforsepa} and Lemma \ref{lem: reducemKbirattoS/G}, all above lower bounds for $c$ or $c-\frac{c^2}{\beta}$ implies $|mK_X|$ defines the Iitaka fibration for any $m\geq 31$ and (1) follows.

So we might assume $h=|G|\leq 68$.
Then $\beta=\frac{t^2}{|G|},c=\frac{5t}{|G|}$ and
\[
 c-\frac{c^2}{\beta}=\frac{5t-25}{|G|}\geq\frac{125}{|G|}.
\]
When $|G|\leq 62$, (1) follows by 
$c-\frac{c^2}{\beta} >2$, Lemma \ref{lem: criterionforsepa} and Lemma \ref{lem: reducemKbirattoS/G}. 
So we can assume $63\leq |G| \leq 68$ and $r=5$, which is exactly (2).
\end{proof}
\begin{proposition}\label{prop: Godsurf}
Let $X\sim_{\text{bir}}(T\times E)/G$ be a smooth generically isotrivial threefold.
Assume $G$ acts on $E$ by translations.
Denote by $S$ the minimal resolution of $T$.

If $K_S^2=1$ and $p_g(S)=q(S)=0$, then $|mK_X|$ defines the Iitaka
fibration for every integer $m\geq 31$.
\end{proposition}
\begin{proof}
Proposition~\ref{prop: ratpencilcase} (6) gives $|G|\leq 70$.
So by Lemma \ref{lem: Godsurf}, we might assume 
\[
K_S\cdot\Gamma=5,\qquad |G|\in\{63,64,65,66,67,68\},
\]
where $\Gamma\subset S$
is a $G$-invariant curve passing through very general points.
Denote by
$\nu:\widetilde{\Gamma}\to\Gamma$ the normalization of $\Gamma$. 
Then $G$ acts faithfully on $\widetilde{\Gamma}$, and
\begin{equation}\label{eq: genus16}
 g(\widetilde{\Gamma}) \geq 2,\qquad
 2g(\widetilde{\Gamma})-2\le K_S\cdot\Gamma+\Gamma^2\le5+25=30.
\end{equation}
by Lemma \ref{lem: curvebound}.
The line bundle $\mathcal L=\nu^*(K_S|_\Gamma)$ is naturally
$G$-linearized in the sense of \cite[Section 2.1]{GSSE2013}, and $\deg(\mathcal L)= K_S\cdot\Gamma=5$.

Since $G\leq \Aut(E)$ for some elliptic curve $E$, either $G$ is cyclic or $|G|=68$ and $G\cong C_{34}\times C_2$ or $|G|=64$.
Moreover,
\cite[Theorem 1.3]{PE08} implies $3\nmid |G|$ and $|G|\in \{ 64,65,67,68 \}$.
So we have the following cases.
\\
\medskip
{\bf Case 1} \ If $G$ is cyclic group with $|G|\in\{ 64,65,67,68 \}$.
Write $b:=g(\widetilde{\Gamma}/G)$ and let $e_1,\ldots,e_k$ be the
branch orders. 
Applying Hurwitz formula on quotient map $\widetilde{\Gamma}\to \widetilde{\Gamma}/G$ and taking degree, we have
\begin{equation}\label{eq: 30gGamma}
    30\geq 2g(\widetilde{\Gamma})-2=|G|\cdot\left(2b-2+\sum_{i=1}^k(1-e_i^{-1})\right), 
\end{equation}
where $|G| \in \{ 64,65,67,68 \}$.

If $b\geq 2$, then the right side is at least $128$, this contradicts $30\geq 2g(\widetilde{\Gamma})-2$. 

If $b=1$, an unramified
cover has genus one, while a ramified cover has right hand side at least
$|G|\cdot(1-1/2)\geq 32 >30$. 
This contradicts $30\geq 2g(\widetilde{\Gamma})-2$.

Hence $b=0$ and $k\geq 3$.
Write monodromies on branch points in $\widetilde{\Gamma}/G\cong \Pp^1$ additively as $a_i\in\mathbb Z/n\mathbb Z$, where $\operatorname{ord}(a_i)=e_i$.
Moreover,
we have
\[
 \langle a_1,\ldots,a_k\rangle=G,\qquad
 \sum_{i=1}^k a_i=0.
\]
For every prime number $p\mid |G|$, at least two branch orders $e_i$ are divisible by $p^r$, where $p^r\mid |G|$ but $p^{r+1}\nmid |G|$.
Indeed, since $a_i$ generates $G$, at least one $a_1$ is non-zero modulo $p$.
By
$\sum a_i\equiv 0 \bmod{p}$, at least two $a_1,a_2$ are non-zero modulo $p$.
Then $p^r| e_j=\frac{|G|}{\gcd(|G|,a_j)}$ for $j=1,2$.

Let $p$ be the largest prime divisor of $|G|$.
There is a positive integer $r$ such that $p^r\mid |G|$ but $p^{r+1}\nmid |G|$.
If $k\geq 4$, then
\[
 2g(\widetilde{\Gamma})-2\ge
 |G|\cdot \left(-2+2\left(1-\frac{1}{p^r}\right)+\frac{k-2}{2}\right)
 \geq |G|\cdot  \left(1-\frac{2}{p^r}\right)\geq 54,
\]
where $(|G|,p^r)\in\{(64,64),(65,13),(67,67),(68,17)\}$ here.

Thus $k=3$. Any two of the $a_i$'s generate $G$, as $a_3=-(a_1+a_2)\in \frac{\Z}{{|G|\cdot\Z}}$. 
So
\[
 e_i\mid |G|,\qquad \operatorname{lcm}(e_i,e_j)=|G|\quad(i\ne j).
\]
By calculations on prime divisors of $|G|\in \{64,65,67,68\}$ respectively, we have
\[
\begin{array}{c|c|c}
 |G| & \text{triple maximizing }\displaystyle\sum_{i=1}^3 e_i^{-1}
   & \text{lower bound for }2g(\widetilde{\Gamma})-2\\ \hline
 64 & (2,64,64)        &   30   \\
 65 &(5,13,65)&46\\
 67 &(67,67,67)&64\\
 68 &(2,68,68)&32
\end{array}
\]
By \eqref{eq: genus16}, all cases except the first row are excluded.
So we might assume $(e_1,e_2,e_3,|G|)=(2,64,64,64)$.

Now equality holds throughout \eqref{eq: genus16} and \eqref{eq: 30gGamma}.
These equalities imply $p_a(\Gamma)=g(\widetilde{\Gamma)}$ and $\Gamma$ is smooth, $\Gamma^2=25$, and 
\[
\Gamma\equiv 5K_S,
\]
by the Hodge index theorem.

Denote by $\sigma$ the unique involution of $G$.
Then every nontrivial subgroup of $G$ contains $\sigma$.
Then the set of fixed points $\operatorname{Fix}(\sigma|_\Gamma)$ of $\sigma$
are precisely the points over the three branch
points, and
\[
\#\operatorname{Fix}(\sigma|_\Gamma)
=\frac{64}{2}+\frac{64}{64}+\frac{64}{64}
=34.
\]

However, write the fixed locus of
$\sigma$ on $S$ as the union of its divisorial
part $R$ and its isolated fixed points.
By \cite[Proposition~4.5(i),(iii)]{CCML}, there
are exactly five isolated fixed points
and
\[
K_S \cdot R=1.
\]
Since the action on $\Gamma$ is faithful,  $\Gamma$ is
not a component of $R$. 
Therefore
\[
\#\operatorname{Fix}(\sigma|_\Gamma)
\leq\Gamma \cdot R+5
=5K_S\cdot R+5
=10,
\]
contradicting the preceding count of $34$.
\\
\medskip
{\bf Case 2} \ If $|G|=64$ and $G$ is non-cyclic, or $|G|=68$ and $G\cong \Z_{34}\times \Z_2$.
For any closed point $p\in \widetilde{\Gamma}$,
the stabilizer $G_p:=\{g\in G|g.p=p\}$ of a faithful finite action $G$ on a smooth complex curve
is cyclic, as the representation on its one-dimensional tangent space is faithful.
Thus cyclic $G_p$ must be a proper subgroup of non-cyclic group $G$.

If $|G|=64$, then $2\mid [G:G_p]$ as $G_p\neq G$ and $|G|=2^6$.
If $|G|=68$ and $G\cong \Z_{34}\times \Z_2$, then $|G_p|\in \{1,2,17,34\}$, and the corresponding orbit lengths lie in
$\{68,34,4,2\}$. So the lengths of orbits are all even in both cases.

Then
Lemma~\ref{lem:linebundleorbit} implies that every
$G$-linearized line bundle has even degree, contrary to the naturally
$G$-linearized line bundle $\mathcal L=\nu^*(K_S|_\Gamma)$ with $\deg(\mathcal L)= K_S\cdot\Gamma=5$.
\end{proof}
\begin{lemma}\label{lem:linebundleorbit}
Let $G$ be a finite group acting faithfully on a smooth projective curve $C$.
Let $\mathcal{L}$ be a $G$-linearized line bundle defined in \cite[Section 2.1]{GSSE2013} on $C$.
Then $d|\deg\mathcal{L}$, where $d$ is common factor of the lengths of all $G$-orbits.
\end{lemma}
\begin{proof}
Since $\mathcal{L}$ is $G$-linearized, by \cite[Lemma 2.1]{GSSE2013}, there is a $G$-invariant divisor $D$ with $\OO_C(D)\cong \mathcal{L}$ so that $\deg D=\deg \mathcal L$ .
Then 
\[
D=\sum_{\mathcal O}n_{\mathcal O}
  \left(\sum_{p\in\mathcal O}p\right),\qquad n_{\mathcal O}\in\mathbb Z,
\]
where $\mathcal O$ runs through the orbits of $G$ supported on $\operatorname{Supp}(D)$, and $n_{\OO}$ is the common multiplicity of $D$ on $\mathcal O$.
Hence 
\[
\deg D=\sum_{\mathcal O}n_{\mathcal O}
  \cdot|\mathcal O|=\left(\sum_{\mathcal O}n_{\mathcal O}
  \cdot\frac{|\mathcal O|}{d}\right)\cdot d
\]
is divisible by $d$.
\end{proof}
\subsection{On $(1,2)$ and $(2,3)$-surfaces}
In this part, we prove Proposition \ref{prop: invariantsectionring-12} and Proposition \ref{prop: invariantsectionring-23}, which together with Proposition \ref{prop: nVbdimply32K} give effective Iitaka fibration for generically isotrivial case with a $(1,2)$ or $(2,3)$-surface factor.

For $u=(w_1,\dots,w_r)\in  \mathbb R_{\geq 0}^r$, write $|u|=w_1+\dots +w_r$.
\begin{lemma}
\label{lem: latticegenerator}
Let $L_0\subset L\subset\mathbb Z^r$ be lattices of rank $r$.
Suppose that
$L_0$ has a basis $u_1,\ldots,u_r\in\mathbb Z_{\geq0}^r$.
Denote $U:=u_1+\cdots+u_r$ with norm $|U|=|u_1|+\cdots+|u_r|$.

Then every coset of $\mathbb Z^r/L_0$ has a nonnegative integral
representative $p$ with $|p|<|U|$.
Moreover, $L$ is generated by $u_1,\ldots,u_r$ together with nonnegative
vectors $s\in L$ satisfying $|s|\leq |U|/2$.
\end{lemma}
\begin{proof}
By 
\cite[Example 3.3]{BP99},
$\Pi:=\sum_i[0,1)u_i$ contains a unique representative $p$ of
each coset of $L/L_0$
with $|p|<|U|$.
Then such $p$ is the desired non-negative representative by taking $L=\Z^r$.

Moreover, for each representative $s\in\Pi\cap L$ with $|s|>|U|/2$,
one has $ |U-s|\leq \frac{|U|}{2}$.
Its class in $L/L_0$ is $-s$.
Replacing $s$ with $U-s$ does not change the generated subgroup.
Together with the $u_i$, the resulting vectors generate $L$.
\end{proof}
\begin{proposition}
\label{prop: invariantsectionring-12}
Let $X\sim_{\text{bir}}(T\times E)/G$ be a smooth generically isotrivial threefold.
Assume $G$ acts on $E$ by translations.
Denote by $S$ the minimal resolution of $T$ with $(K_S^2,p_g(S),q(S))=(1,2,0)$.

If $|G|<100$, then $|mK_X|$ defines the Iitaka fibration for any $m\geq 31$.
\end{proposition} 
\begin{proof}
Since $T$ is $(1,2)$-surface, we can represent $T$ as hypersurface $\{w^2=F_{10}(x,y,z)\} \subset \Pp(1,1,2,5)$ as described in Lemma \ref{lem: whenT12G100} (1).
By  Lemma \ref{lem: PS17}, 
we may assume $G$ acts on $x,y,z,w$ by scalars, and write the action of $G$ as
\[
 g[x:y:z:w]=[\alpha_gx:\beta_gy:z:w],\qquad \alpha_g,\beta_g\in \C^*.
\]
Then $g.\Omega_S=\alpha_g\beta_g \cdot\Omega_S$, where $\Omega_S$ is the canonical form on smooth minimal model $S$.

There are two linearly independent nonzero exponent vectors $u,v$
among monomials occurring in $F_{10}(x,y,z)$. 
Indeed, if not, denote $u:=(a,b)\in \Z^2$.
Consider $(w^2-F_{10})$-invariant actions
\[
\lambda:[x:y:z:w] \mapsto [\lambda^bx:\lambda^{-a}y:z:w], \qquad \lambda \in \C^*, 
\]
which generates an infinite subgroup of $\Aut(S)$.
This contradicts the finiteness of $\Aut(S)$ for general type $S$.
Denote $L_0:=\Z u+\Z v$ with $|u|,|v|\leq 10$.

Define the lattice
\[
 L=\{(a,b)\in\mathbb Z^2:
             \alpha_g^a\beta_g^b=1\text{ for all }g\in G\}.
\]
Define
\[
 \Lambda=\{(a,b)\in\mathbb Z^2:a+b\equiv0\pmod2\},
 \qquad K=L\cap\Lambda.
\]
Then $L_0\subseteq K\subseteq L$.

If $L_0\subsetneq L$, then
\[
|G|=[\Z^2:L]\leq \frac{1}{2}[\Z^2:L_0] \leq 50,
\]
where the last inequality holds by the proof of Lemma \ref{lem: whenT12G100} (2).
Then $|31K_X|$ defines the Iitaka fibration by Lemma \ref{lem: Godsurf}.
So we might assume $K=L=L_0$, $L\subseteq \Lambda$ and $L$ is generated by $u,v\in L_0$.

By definition of $L$, $\C(S)^{G}=\C\bigl(x^{a}y^{b}z^cw^d|a,b,c,d\in \Z,a+b+2c+5d=0,(a,b)\in L\bigr)$.
Since $2\mid a+b$ by $L=K\subseteq \Lambda$, modulo $2$ one has $2\mid d$. 
By relation $w^2=F_{10}(x,y,z)$,  
\[\C(S)^{G}=\C\left(x^{a}y^{b}z^c(F_{10}(x,y,z))^{\frac{d}{2}}\right)=\C\left(x^{a'}y^{b'}z^{c'}\right)\]
where $a'+b'+2c'=0$.
Moreover, $(a',b')\in L$ as both $(a,b)$ and exponents of monomials appearing in $F_{10}(x,y,z)$
belong to $L$.
Replacing $a,b,c$ by $a',b',c'$, we might assume $\C(S)^{G}\subseteq\C(x^{a}y^{b}z^c|a+b+2c=0,(a,b)\in L)=\C\bigl(x^{a}y^{b}z^{-\frac{(a+b)}{2}}\bigr|(a,b)\in L) $.

We also claim that, for any $m\geq 31$ there is
$p\in\mathbb Z_{\ge0}^2$ with $p+(m,m)\in L_0$ such that
\begin{equation}\label{eq: |p|bd}
|p|\leq \begin{cases}
 16  & \text{ if }m=31; \\
18   & \text{ if }m\geq 32.\\
\end{cases}  
\end{equation}
Since $L=L_0$ is generated by $u,v$ and $|u|,|v|\leq 10$,
there is a representative $p$ of
$(-m,-m)+L_0$ with
\[
|p|<|u|+|v|\leq 20
\]
by Lemma \ref{lem: latticegenerator}.
As $2\mid |p|$, we have $|p|\leq 18$ for any $m$.

Then we assume $m=31$ and prove $|p|\leq 16$ by discussing $|u|,|v|$. 
If $|u|+|v|\leq 18$, then there is a representative $p$ of
$(-31,-31)+L_0$ with
\[
|p|<|u|+|v|\leq 18
\]
by Lemma \ref{lem: latticegenerator}.
Moreover, $|p|$ is even, so $|p|\leq 16$.

Otherwise $|u|=|v|=10$. After interchanging
$u,v$, write
\[
u=(a,10-a),\qquad
v=(a+d,10-a-d),\qquad 1\leq d\leq10.
\]

If $d=10$, then $|G|=[\Z^2:L]=[\Z^2:L_0]=100$.
This contradicts assumption.
Thus $d \leq 9$.
There is an integer $0\leq r<d\leq 9$ satisfying
$r\equiv 7a-31\pmod d$.
Set
$p:=(r,8-r)$.
Since $r\leq 8$, this vector is nonnegative, and
\[
p+(31,31)
=
7u+\frac{r+31-7a}{d}(v-u)\in L_0=L.
\]
So we have proved the claim.

For every $s\in\mathcal \{0,u,v\}\subseteq \Z^2$, define 
\begin{equation*}
 M_s=x^{p_1+s_1}y^{p_2+s_2}
      z^{(m-5\delta-|p+s|)/2}w^\delta.             
\end{equation*}
Since $m\equiv \delta\bmod 2$, exponent of $z$ above is an integer. 
 By equation \eqref{eq: |p|bd} bounding $|p|$ and calculations,
\[
|p+s|\leq |p|+|s| \leq |p|+10 \leq m-5\delta,\]
where  $m\geq 31$, $\delta\in\{0,1\}$ with $m\equiv \delta \bmod 2$, and $|s|\leq 10$.
Thus all exponents in $M_s$ are nonnegative.
Moreover, $M_s$ is $G$-invariant as
$p+(m,m)\in L$ and $s\in \{0,u,v\}$.
Thus $M_s\in H^0(S,mK_S)^G $.

Consider the ratios
\begin{equation*}
 \frac{M_s}{M_0}=x^{s_1}y^{s_2}z^{-|s|/2}\in \mathbb C(S)^G .                 \label{eq:12-base-ratios}
\end{equation*}
Since $s \in \{0,u,v\}$ which generates $L$, we finally obtain
\begin{equation*}
 \mathbb C(S)^G
  \subseteq\mathbb C\bigl(x^{s_1}y^{s_2}z^{-|s|/2}|s\in \{0,u,v\}\bigr)=\C \left( \left\{ \frac{M_s}{M_0} \right\}_{s\in \{0,u,v\}}\right) \subseteq  \mathbb C(S)^G.    
\end{equation*}
Hence the ratios of sections in $H^0(S,mK_S)^G$ generate $\C(S)^G$ and $|mK_X|$ defines the Iitaka fibration  for any $m\geq 31$ by Lemma \ref{lem: reducemKbirattoS/G} (2).

\end{proof}

\begin{proposition}
\label{prop: invariantsectionring-23}
Suppose $(K_S^2,p_g(S),q(S))=(2,3,0)$.
For every finite abelian subgroup $G\leq \operatorname{Aut}(S)$ and
every integer $m\geq 31$, ratios of sections in
$H^0(S,mK_S)^G$ generate $\mathbb C(S)^G$.
\end{proposition}
\begin{proof}

By \cite[Lemma 1.1]{HE76},
we can write the canonical model $T$ and section ring of $S$ as
\[
 T=\{w^2=F_8(x_0,x_1,x_2)\}\subset\Pp(1,1,1,4),\qquad
 R(S,K_S)=\C[x_0,x_1,x_2,w]/(w^2-F_8).
\]
By Lemma \ref{lem: PS17}, we might assume
 $G$ acts on coordinates of affine subset $U:=\{w\neq 0\}\cap T$ by scalars.

Consider the degree-$4$ cover
\[
u_4: \widetilde{U}=\{F_8(x_0,x_1,x_2)=1\}\longrightarrow
       U, \qquad x\longmapsto[x:1].
\]
between affine charts,
and the finite abelian group
\[
 \widetilde G=\{(g,t)\in G\times\mathbb C^*\mid t^4=g(w)\}.
\]
acting on $\widetilde{U}$ corresponding to $G$.
Then $u_4$ induces isomorphisms $\widetilde G/(\Z/4\Z) \cong G$,  $\C(\widetilde{U})^{\widetilde G}\cong \C(U)^{G}=\C(T)^{G}$.
So it suffices to show that $u_4^*(H^0(S,mK_S)^G)$ generates $\C(\widetilde{U})^{\widetilde G}$.

Rescaling the weighted coordinates by
$t^{-1}$, one can write the action for all $(g,t)\in \widetilde{G}$ as
\[
 (x_0,x_1,x_2)\longmapsto
  (\alpha_0x_0,\alpha_1x_1,\alpha_2x_2),
 \qquad \alpha_i=g(x_i)\cdot t^{-1}.
\]

Define
\[
 \Lambda=\{a\in\Z^3:|a|\equiv0\pmod4\},\qquad
 L=\{a\in\Z^3|\alpha_0^{a_0}\alpha_1^{a_1}\alpha_2^{a_2}=1 \text{ for all }(g,t)\in \widetilde{G}\}.
\]
Then $\C(\widetilde{U})^{\widetilde G}=\C\bigl(x_0^{a_0}x_1^{a_1}x_2^{a_2}|(a_0,a_1,a_2)\in L\bigr)$, $L\subseteq \Lambda$, and a monomial $x^a=u_4^*(x^aw^d) \in u_4^*(H^0(S,mK_S)^G)$ of degree $m$, if and only if 
\begin{equation}
 a+m(1,1,1)\in L.                          \end{equation}

Let $\mathcal A$ be the set of exponent vectors of monomials occurring in $F_8$. 
Then $\mathcal A\subset L$. 
Moreover, $\mathcal A$ spans $\mathbb R^3$,
otherwise its common annihilator in the diagonal torus $(\C^*)^3$ would contain a
positive-dimensional torus acting faithfully on $T$, contradicting
finiteness of $\Aut(T)$.

Choose linearly independent $v_1,v_2,v_3\in\mathcal A$ and set
\[
 L_0=\sum_i\Z v_i
\]
where $|v_i|=8$ for each $i$.

Fix $m\geq 31$, set $e=(1,1,1)$, and choose
$p\in -me+L_0$ with $|p|\leq 23<24$ by Lemma \ref{lem: latticegenerator}. 
Then $|p|+3m\equiv0\pmod8$.
So we have $4\mid m-|p|$.
If $m\geq 32$, then $m-|p|\geq 32-23=9$ and $m-|p|\geq 12$ by $4\mid m-|p|$.
If $m=31$, then $|p|+93\equiv0\pmod8$ and $|p|\leq 19<24$. 
So $m-|p|\geq 31-19= 12$.

By Lemma \ref{lem: latticegenerator} applied on $L/L_0$ and $|v_i|=8<12$, there is a finite subset $\mathcal{S}$ generating $L$, with $|s|\leq 12$ where $s\in \mathcal{S}$.
Replacing $\mathcal{S}$ by $\mathcal{S} \cup \{0\}$, we may assume $0\in \mathcal{S}$.

For any $s\in \mathcal{S}$, set
\[
 M_s:=x^{p+s}w^{(m-|p+s|)/4},\qquad s\in\mathcal S.
\]
Since $m-|p|\geq 12\geq |s|$, $M_s$ is nonzero $m$-canonical section with nonnegative exponents.
Since $p+s+me\in L$, $M_s$ is $G$-invariant and $M_s\in H^0(S,mK_S)^G$.

For any $s\in \mathcal{S}$, consider
\[x^s=u_4^*\left(\frac{x^s}{w^{\frac{|s|}{4}}}\right)=u_4^*\left( \frac{M_s}{M_0} \right)\in u_4^*\bigl(\C\bigl(H^0(S,mK_S)^G\bigr)\bigr).\]
Since $s\in \mathcal{S}$ generates $L$,
$\left\{u_4^*(\frac{M_s}{M_0}): s\in \mathcal{S}\right\}$ generates $\C(\widetilde{U})^{\widetilde{G}}=\C(U)^G=\C(T)^G$.
\end{proof}

\subsection{On $(1,1)$-surfaces}
In this part, we prove Proposition \ref{prop: pg1-volume1-abelian}, which together with Lemma \ref{lem: Godsurf} gives effective Iitaka fibration for generically isotrivial case with a $(1,1)$-surface factor.
\begin{proposition}\label{prop: pg1-volume1-abelian}
Let $S$ be a smooth minimal projective surface of general type with
$V=p_g=1$ and $q=0$. 
Let $G\leq \Aut(S)$ be a finite abelian subgroup. 
Then $|G|\leq 36$.
\end{proposition}
\begin{proof}
Let $T=\operatorname{Proj}R(S,K_S)$ be the canonical model of $S$.
The classification of canonical model of $(1,1)$-surfaces
(See \cite[Proposition 1.6 and Theorem~1.7]{Catanese80} for example) implies
\begin{equation}\label{eq:pg1-volume1-ring}
\begin{aligned}
 R(S,K_S)&\simeq\mathbb C[x,y_1,y_2,z_1,z_2]/(F_1,F_2),\\
 \deg(x,y_1,y_2,z_1,z_2)&=(1,2,2,3,3),\qquad
 \deg F_1=\deg F_2=6.
\end{aligned}
\end{equation}

By Lemma \ref{lem: PS17}, we can choose the five generators $x,y_1,y_2,z_1,z_2$ of $R(S,K_S)$ to be semi-invariant vectors of $G$.
The degree-six part of the defining ideal is a two-dimensional $G$-invariant vector space.
Since $G$ is finite and abelian, we may choose an eigenbasis $F_1,F_2$ of this space so that they are semi-invariant.

Take the affine chart $U:= \{x\neq 0  \}\cap T$.
Write polynomial $f_i(u):=F_i(1,u)$.
Let $S_i\subseteq \Z_{\geq 0}^4$ be the set of exponents of monomials $y_1,y_2,z_1,z_2$ occurring in $f_i$.
Define
\[
 L=\left\langle a-b:
       a,b\in S_i\text{ for some }i\in\{1,2\}\right\rangle
       \subseteq\mathbb Z^4.
\]
Since $f_i$ is semi-invariant as $G$ acts on $F_i$ by scalars, the image of $G$ is contained in
\[
 D_L=\{\lambda\in(\mathbb C^*)^4:
                   \lambda^\ell=1\text{ for every }\ell\in L\}.
\]

The $D_L$-action on $U$ is also faithful, and
\[
 \operatorname{rank}L=4,
 \qquad
 |D_L|=[\mathbb Z^4:L].
\]
The last equality follows directly from
$D_L\simeq\operatorname{Hom}(\mathbb Z^4/L,\mathbb C^*)$. 

Since $\deg F_1=\deg F_2=6$, every vector in
$S_1\cup S_2$ belongs to the weighted simplex
\[
 P=\{a\in\mathbb R_{\geq0}^{4}:
                2a_1+2a_2+3a_3+3a_4\leq6\}.
\]
So for any integral matrix $M$ generated by four linearly independent vectors in $(S_1-S_1)\cup (S_2-S_2)\subseteq P-P$, one has
\[
 |\det M|\leq\frac{6^4}{2\cdot2\cdot3\cdot3}=36
\]
by Schinzel's determinant inequality
(See \cite{Schinzel77}).
Hence we have
\begin{equation*}
 |G|\leq |D_L|= [\mathbb Z^4:L]
       \leq [\mathbb Z^4:L_0]=|\det M|\leq 36,
\end{equation*}
where $L_0\subseteq  L$ is a sublattice spanned by generators of $M$.
\end{proof}

\addtocontents{toc}{\protect\setcounter{tocdepth}{2}}
\section{Examples}
The following example shows the sharpness of Theorem \ref{mainthm: eff-iitaka} (1) and Theorem \ref{mainthm: iitaka-vol} (1).
\subsection{Sharpness of $1/100$ and $32K$}
\label{subsec: 32K}
\begin{example}\label{ex: 32Kissharp}
Let
\[
T
=
\left\{
w^2=z^5+x^{10}+y^{10}
\right\}
\subset
\mathbb P(1,1,2,5).
\]
and the defining equation has weighted degree $10$.

Consider the diagonal quotient
\[
X:=(T\times E)/(\Z_{10} \times \Z_{10}),
\]
where $g=(s,t)\in\Z_{10} \times \Z_{10}$ acts on $T\times E$ by 
\[
g\cdot([x:y:z:w],e)
:=
([\zeta_{10}^sx:\zeta_{10}^ty:z:w],g.e).
\]
where $g.e$ is corresponding translation of the isomorphism $\Z_{10} \times \Z_{10}\cong E[10]$.

Then it is easy to see that $\pi:T\times E \to X$ is an \'etale quotient so that $X$ is smooth $3$-fold with $\kappa(X)=2$ and $\Iv(X)=100^{-1}\Iv(T\times E)=\frac{1}{100}$.
Moreover, $\Z_{10} \times \Z_{10}$ acts on the second factor $E$ by translations and acts trivially on one-forms of $E$.
So the Albanese morphism of $X$ does not contract $E$.
Since $E$ is general fiber of the Iitaka fibration, the Albanese morphism of $X$ does not factor through the Iitaka fibration.

We can also calculate the canonical ring of $X$ as follows.
\end{example}

\begin{lemma}\label{lem: listinvform}
The section ring $R(X,K_X)$ is the following free graded $B$-module of rank $10$
\begin{equation}\label{eqsectionringX}
 R(X,K_X)=\bigoplus_{s=0}^{9}h_sB,   
\end{equation}
where $B:=\C [x^{10},y^{10},z^5]$, and each monomial $h_s$ appears in linear system $|mK_X|$ when $m\leq 31$, as follows
\begin{equation} \label{lisths}
\begin{array}{c|c|c} 
\deg_K h_s \pmod{10} & h_s & \deg_K h_s\\ \hline
0 & 1 & 0\\
1 & x^9y^9wz^4 & 31\\
2 & x^8y^8z^3 & 22\\
3 & x^7y^7wz^2 & 23\\
4 & x^6y^6z & 14\\
5 & x^5y^5w & 15\\
6 & x^4y^4z^4 & 16\\
7 & x^3y^3wz^3 & 17\\
8 & x^2y^2z^2 & 8\\
9 & xywz & 9
\end{array}
\end{equation}
\end{lemma}
\begin{proof}
It is easy to see that
\[
H^0(X,mK_X)=(H^0(T,mK_T)\otimes H^0(E,mK_E))^{\Z_{10} \times \Z_{10}}=
  H^0(T,mK_T)^{\Z_{10} \times \Z_{10}}.
\]
as $mK_E$ is invariant under translations.

Denote the canonical form of $\Pp(1,1,2,5)$ by
\[
\begin{aligned}
\Omega={}&
 x\,dy\wedge dz\wedge dw
-y\,dx\wedge dz\wedge dw \\
&+2z\,dx\wedge dy\wedge dw
-5w\,dx\wedge dy\wedge dz.
\end{aligned}
\]
By adjunction formula, the canonical form of $T$ is 
\[
\rho=\operatorname{Res}_{T}\!\left(\frac{\Omega}{F}\right),
\]
where $F=-w^2+z^5+x^{10}+y^{10}$ is $\Z_{10} \times \Z_{10}$-invariant.

Then $|mK_T|$ is generated by $\{P\cdot\rho^{\otimes m}\}$, where $P=P(x,y,z,w)$ is a monomial with $\deg(P)=m$.

By definition, $(s,t)\in \Z_{10} \times \Z_{10}$ acts on $\rho$ by
\[
(s,t)^*\rho=(\zeta_{10}^s\zeta_{10}^{t})\rho.
\]
Then, for a monomial
$P=x^a y^b z^c w^d$
with
$a+b+2c+5d=m\leq 31$.
The associated $m$-canonical form $P\cdot\rho^{\otimes m}$ is $\Z_{10} \times \Z_{10}$-invariant if and only if
\begin{equation}\label{eqabequiv-mmod10}
\begin{cases}
a+m \equiv 0 \pmod{10} \\
b+m\equiv0\pmod{10}
     \\
m=a+b+2c+5d.   
\end{cases} 
\end{equation}
or  
\begin{equation}\label{eqabcdequiv-mmod10}
\begin{cases}
2a+b+2c+5d \equiv 0 \pmod{10} \\
a+2b+2c +5d\equiv 0\pmod{10}
\end{cases}
\end{equation}
as $m=a+b+2c+5d$.
Then we obtain that such monomial $P$ is generated by elements listed in \eqref{lisths} together with $x^{10}, y^{10}, z^5$ by \eqref{eqabcdequiv-mmod10}.

Moreover, for each $m\geq 32$, consider monomial $P':=x^{a'} y^{b'} z^{c'} w^{d'} $ with $\deg P'=m$ and $\Z_{10} \times \Z_{10}$-invariant $P'\cdot \rho^{\otimes m}$.
Let $k_{a'},k_{b'},k_{c'},k_{d'}$ be maximal integers such that 
\[
\begin{cases}
0 \leq  a'-10k_{a'} \leq 9\\
0\leq b'-10k_{b'} \leq 9 \\
0\leq c'-5k_{c'} \leq  4   \\
0\leq d'-2k_{d'} \leq 1.                             
\end{cases}
\]
then $Q':=x^{a'-10k_{a'} } y^{b'-10k_{b'} } z^{c'-5k_{c'} } w^{d'-2k_{d'} } $ is $\Z_{10}\times \Z_{10}$-invariant with $\deg Q'\leq 9+9+4\times2+5=31$ so that $Q'$ is generated by the monomials $h_s$ listed in \eqref{lisths} and $x^{10}, y^{10}, z^5$.
Hence $P'$ is also generated by these monomials and we have proved \eqref{eqsectionringX}.

In particular, we have
\[
H^0(X,14K_X)=
\mathbb C\cdot x^6y^6z
\]
so that $|14K_X|$ is constant and it does not define an Iitaka fibration.     
\end{proof}
\begin{lemma}\label{lem: iitakabaseX}
We can calculate the $10$-th Veronese subring as
\begin{equation} \label{eq10qKX}
R(X,K_X)^{(10)}
=
\bigoplus_{q\geq0}H^0(X,10qK_X)
\cong
\C [x^{10},y^{10},z^5].
\end{equation}
Then the Iitaka fibration of $X$ is
\[
f:X\longrightarrow\mathbb P^2,
\qquad
[s,e]\longmapsto
[x^{10}(s):y^{10}(s):z^5(s)].
\]    
\end{lemma}
\begin{proof}
Equation \eqref{eq10qKX} follows from Lemma \ref{lem: listinvform}.

For $q$ sufficiently large and divisible, $|10qK_X|$ defines an Iitaka fibration by \cite[Theorem 1.2]{BirkarZhang16}.
So the Iitaka base is $\Pp^2$ by  \eqref{eq10qKX}.   
\end{proof}

\begin{proposition}\label{prop: 31KdoesnotdefIitfib}
In fact, $|mK_X|$ gives an Iitaka fibration, if and only if
\[
 m\in
 \{10,18,19,20,24,25,26,27,28,29,30\}
 \cup\{m\in\mathbb Z:m\geq32\}.
\]    
\end{proposition}
\begin{proof}
By Lemma \ref{lem: iitakabaseX}, the Iitaka fibration of $X$ is determined by degree $10$ monomials $x^{10},y^{10},z^5$.
By Lemma \ref{lem: listinvform}, $|mK_X|$ defines the Iitaka fibration, if and only if there is $h_s$ in list \eqref{lisths} such that $m-\deg_K h_s=10k$ for some $k\in\N_{\geq 1}$.
Then the proof is completed by list \eqref{lisths}.
\end{proof}

\numberwithin{equation}{section}

\subsection{Sharpness of $14K$ in the Albanese-factoring case}
The following example shows the sharpness of Theorem \ref{mainthm: eff-iitaka} (2).
\begin{example}\label{ex: 14Kissharp}
    We construct a smooth projective threefold $X$ of Kodaira dimension two, which is birational to \[(D \times W)/G,\] where $D$ is a curve of genus $2$, $W$ is an elliptic modular surface with $\kappa(W)=1$ and $G$ acts diagonally. Moreover, the Iitaka fibration of $X$ is not generically isotrivial and $\lvert 13K_X \rvert$ does not define the Iitaka fibration.

    Let $D$ be the smooth projective model of \[y^2=x^6-1,\] which has genus two. Define $G$ as \[<\sigma, \tau>\cong \mathbb{Z}/6\mathbb{Z}\oplus \mathbb{Z}/2\mathbb{Z},\; \sigma(x,y)=(\zeta_6 x, y), \; \tau(x,y)=(x,-y),\] where $\zeta_6$ is a primitive sixth root of unity. 

    Let \[\Gamma(6):=\{\gamma\in \operatorname{SL}_2(\mathbb{Z})\vert \gamma\equiv \operatorname{Id}\; (\operatorname{mod}6)\}\] be the principal congruence subgroup of level $6$. According to \cite[Definition 4.1]{Shioda72}, one can obtain an elliptic surface fibration \[\pi: W\to B\] attached to $\Gamma(6)$. By \cite[Theorem 5.5]{Shioda72}, the group of its sections is isomorphic to $\mathbb{Z}/6\mathbb{Z}\oplus \mathbb{Z}/6\mathbb{Z}$. Let $P$ and $Q$ be two generators. By \cite[Example 5.4]{Shioda72}, the genus of $B$ is equal to one and $\pi$ has exactly $12$ singular fibers, all of type $I_6$. \cite[Theorem 6.10, Corollary 6.11]{SchuettShioda2010} give 
    
    \begin{equation} \label{eqchiW6}
        e(W)=72, \qquad \chi(\OO_W)=\frac{e(W)}{12}=6.
    \end{equation}
    
    By \cite[Section 7.6]{SchuettShioda2010}, $P$ and $Q$ give automorphisms $t_P$ and $t_Q$ of $W$ preserving $B$ and both have order $6$. Thus, we can give an embedding of $G\hookrightarrow \operatorname{Aut}(W)$ by 
    \[\sigma\mapsto t_P, \; \tau \mapsto t_{3Q}.\]

    Hence, $X_0:=(D\times W)/G$ is a normal projective threefold. Choose a resolution $X\to X_0$ that is an isomorphism over the smooth locus. By \cite[Lemma 2.1]{MirandaPersson1989}, every non-identity element of $G$ has finitely many fixed points on $D\times W$ so $p: D\times W\to (D\times W)/G$ is \'{e}tale in codimension one. Thus, we have $p^\ast K_{X_0}=K_{D\times W}$ and $K_{X_0}$ is nef. By Reid-Tai Criterion in \cite[(4.11)]{Reid1987}, one can check that $(D\times W)/G$ has terminal cyclic quotient singularities. 
\end{example}
\begin{lemma}\label{lem: h0(D)^G}
    The quotient map $\tau: D\to D/G$ identifies $D/G$ with $\Pp^1$. Moreover, for every $m\geq 1$, 
    \[H^0(D, mK_D)^G\cong H^0(\Pp^1, \omega_{\Pp^1}^{\otimes m}(\lfloor \frac{5m}{6}\rfloor [0]+\lfloor\frac{m}{2} \rfloor [1]+\lfloor \frac{5m}{6} \rfloor [\infty])).\]
    In particular, its dimension is equal to $\max\{d_m+1, 0\}$, where \[d_m=-2m+\lfloor \frac{m}{2} \rfloor +2\lfloor \frac{5m}{6} \rfloor.\]
\end{lemma}
\begin{proof}
One can easily check that $\mathbb{C}(D)^G=\mathbb{C}(x^6)$ and the branch values are $0,1,\infty$, with ramification indices $6, 2, 6$, respectively. 

By the Hurwitz formula, we have
\[K_D=\tau^\ast(K_{\mathbb P^1}+\Delta),\quad \Delta:=\frac{5}{6}[0]+\frac{1}{2}[1]+\frac{5}{6}[\infty]. \]
Thus,
\[\begin{aligned}
\bigl(\tau_\ast\OO_D(mK_D)\bigr)^G &\cong \left(\tau_\ast\OO_D\left(\tau^\ast \left(mK_{\mathbb P^1}+\frac{5m}{6}[0]+\frac{m}{2}[1]+\frac{5m}{6}[\infty]\right)\right)\right)^G\\
&\cong\OO_{\Pp^1}\left(mK_{\mathbb P^1}+\left\lfloor\frac{5m}{6}\right\rfloor[0]+\left\lfloor\frac{m}{2}\right\rfloor[1]+\left\lfloor\frac{5m}{6}\right\rfloor[\infty]\right).
\end{aligned}\]
Taking global sections, we obtain
\[H^0(D,mK_D)^G\cong H^0\left(\Pp^1, \OO_{\Pp^1}\left(mK_{\Pp^1}+\left\lfloor\frac{5m}{6}\right\rfloor[0]+\left\lfloor\frac{m}{2}\right\rfloor[1]+\left\lfloor\frac{5m}{6}\right\rfloor[\infty]\right)\right).\]
The divisor on the right has degree
\[d_m=-2m+\left\lfloor\frac{m}{2}\right\rfloor+2\left\lfloor\frac{5m}{6}\right\rfloor\]
and we conclude by the Riemann-Roch theorem.
\end{proof}
\begin{lemma}\label{lem:h^0(W)^G}
    For every $m\geq 1$, $G$ acts trivially on $H^0(W, mK_W)$.
\end{lemma}
\begin{proof}
    Let $W_\eta$ be the generic fiber of $\pi$. We have 
    \[H^0(W, mK_W)\hookrightarrow H^0(W_\eta, (mK_W)\vert_{W_{\eta}})=H^0(W_\eta, \omega_{W_{\eta}}^{\otimes m})\otimes (\Omega^1_{\mathbb{C}(B)/\mathbb{C}})^{\otimes m}.\]
    Thus, we conclude by \cite[Chapter III, Proposition 5.1]{Silverman2009}.
\end{proof}
\begin{lemma}\label{lem: h0=G+h}
    For $m\geq 1$, there is a natural identification \[H^0(X, mK_X)\cong H^0(D, mK_D)^G \otimes H^0(B, L^{\otimes m}), \]
    where $L:=\pi_\ast (\omega_{W/B})$.
\end{lemma}
\begin{proof}
Note that $L$ is a line bundle. By canonical bundle formula and \eqref{eqchiW6}, $\operatorname{deg}L=\chi(\OO_W)=6$. We have $\omega_W\cong \pi^\ast (\omega_{B}\otimes L)\cong \pi^\ast L$, since $\pi$ has no multiple fibers. Thus, for every $m\geq 1$, \[H^0(W, mK_W)\cong H^0(B, L^{\otimes m}), \qquad P_m(W)=6m.\]

    Since $D\times W\to X_0$ is \'{e}tale in codimension one and $X_0$ has canonical singularities, we have \[H^0(X,mK_X)\cong H^0(X_0, \omega_{X_0}^{[m]})\cong H^0(D\times W, mK_{D\times W})^G,\] which is isomorphic to \[H^0(D, mK_D)^G \otimes H^0(W, mK_W)\cong H^0(D, mK_D)^G \otimes H^0(B, L^{\otimes m})\]
    by Lemma \ref{lem:h^0(W)^G}.
\end{proof}
\begin{proposition}
    There is a factorization \[X\xrightarrow{f} \Pp^1\times B \xrightarrow{\operatorname{pr_2}} B,\]
     such that $f$ is birationally equivalent to the Iitaka fibration and $a:=\operatorname{pr}_B \circ f$ is the Albanese map of $X$. 

     Moreover, $\lvert 13K_X\rvert $ does not define the Iitaka fibration.
\end{proposition}
\begin{proof}
    Since $(D\times W)/G$ is terminal, it has rational singularities so \[H^1(X, \OO_X)\cong H^1((D\times W)/G, \OO_{(D\times W)/G}).\] By K\"unneth formula, we have \[ H^1(X,\OO_X)\cong H^1(D\times W, \OO_{D\times W})^G\cong H^1(D, \OO_D)^G\oplus H^1(W, \OO_W)^G,\]
    where the first summand is isomorphic to $H^1(D/G, \OO_{D/G})=0$.
    By \cite[Theorem 6.12]{SchuettShioda2010}, pullback induces an isomorphism \[\pi^\ast: H^1(B, \OO_B)\xrightarrow{\sim} H^1(W, \OO_W).\] For every $g\in G$, we have $\pi \circ g=\pi$ so $g^\ast \circ \pi^\ast =\pi^\ast $. Hence, $g$ acts trivially on $H^1(W, \OO_W)$ and $H^1(W,\OO_W)^G$ is isomorphic to $\mathbb{C}$, which implies that $q(X)=1$.

    Note that $f$ is induced by the $G$-invariant map $D\times W \to \Pp^1\times B$. By the construction $f$ has connected fibers so the composition $X\to B$ is also a fibration. Note that $B$ and $\operatorname{Alb}(X)$ are both elliptic curves. By the universal property of Albanese map and the connectedness of the fibers of $X\to B$, $\operatorname{pr}_B \circ f$ must be the Albanese map.

    By Lemma \ref{lem: h0=G+h} and Lemma \ref{lem: h0(D)^G}, if $d_m\geq 0$, the $m$-canonical map $\varphi_{\lvert mK_X\rvert}$ is the composition of $f$ and the map of $\Pp^1\times B$ defined by $\OO_{\Pp^1}(d_m)\boxtimes L^{\otimes m}$. 

    For $m=13$, we obtain $d_{13}=0$ and $P_{13}=78$. The first factor of $\OO_{\Pp^1}(d_m)\boxtimes L^{\otimes m}$ is trivial and the second factor is very ample on the elliptic curve. Thus, $\varphi_{\lvert 13K_X\rvert}$ is just $a=\operatorname{pr}_B\circ f$.

    For $m\geq 14$, we obtain $d_{m}\geq 1$ and $\operatorname{deg}L^{\otimes m}=6m$. Thus, $\OO_{\Pp^1}(d_m)\boxtimes L^{\otimes m}$ is very ample and $\varphi_{\lvert mK_X\rvert}$ is just $f$. Thus, $f$ is birationally equivalent to the Iitaka fibration.
\end{proof}

\addtocontents{toc}{\protect\setcounter{tocdepth}{2}}
\section{Proof of the main theorems}
In this section, we prove Theorems \ref{mainthm: eff-iitaka}, \ref{mainthm: iitaka-vol}, \ref{mainthm: eqcaseiveff} and \ref{mainthm:translations}.

\begin{proof}[Proof of Theorem {\ref{mainthm: eff-iitaka}}]
If the Albanese morphism of $X$ factors through the Iitaka fibration, then $|mK_X|$ defines the Iitaka fibration for $m\geq 14$ by Theorem \ref{thm: 14K}.

If the Albanese morphism of $X$
does not factor through its Iitaka fibration, $X$ is generically isotrivial by Corollary \ref{cor: chargeneisotrivial3fold} $(2)$.
Moreover, $G$ is abelian, and it acts on $E$ by translations by Corollary \ref{cor: chargeneisotrivial3fold} (3).

Hence $X$ satisfies all the assumptions in Theorem \ref{thm: geneisotrivialeffIitakafib} and $|mK_X|$ defines the Iitaka fibration for any $m\geq 32$, and $|31K_X|$ does not define the Iitaka fibration if and only if $T$ and
$G$ satisfy the conclusions in Lemma \ref{lem: whenT12G100} (2).

The sharpness of $32$ and $14$ follows from Example \ref{ex: 32Kissharp} and \ref{ex: 14Kissharp} respectively.
\end{proof}

\begin{proof}[Proof of Theorem \ref{mainthm: iitaka-vol}]
If the Albanese morphism of $X$ factors through the Iitaka fibration, then $\Iv(X)\geq \frac{1}{24}$ by Theorem \ref{thm:non-iso-vol}.

If the Albanese morphism of $X$ does not factor through the Iitaka fibration, $X\sim_{\text{bir}} (T\times E)/G$ is generically isotrivial by Corollary \ref{cor: chargeneisotrivial3fold} (2).

Thus $X$ satisfies the assumptions in Theorem \ref{thm: geneisotrivialvollowerbd} and $\Iv(X)\geq \frac{1}{100}$.

Moreover, sharpness of $\frac{1}{100}$ follows from Example \ref{ex: 32Kissharp}.
\end{proof}

\begin{proof}[Proof of Theorem {\ref{mainthm: eqcaseiveff}}]
If the Albanese map of $X$ factors through the Iitaka fibration,
then $\operatorname{Ivol}(X)\geq 1/24>1/100$ and $|31K_X|$ defines the Iitaka fibration
by Theorem \ref{mainthm: iitaka-vol} (2) 
and Theorem \ref{mainthm: eff-iitaka} (2) respectively.
Otherwise, Corollary \ref{cor: chargeneisotrivial3fold} (3) applies and we might assume $X$ is generically isotrivial.

Then Condition (1) $\Leftrightarrow$ (4) follows by Theorem \ref{thm: geneisotrivialvollowerbd} (2).
Condition  (2) $\Leftrightarrow$ (3) $\Leftrightarrow$ (4) follows by Theorem \ref{thm: geneisotrivialeffIitakafib} (2).
\end{proof}

\begin{proof}[Proof of Theorem {\ref{mainthm:translations}}]
    It follows directly from Theorem \ref{thm: geneisotrivialvollowerbd}.
\end{proof}

\bibliographystyle{alpha}
\bibliography{reference.bib}
\end{document}